\documentclass{article}

\usepackage{titlesec}
\titleformat{\subsection}[runin]
       {\normalfont\bfseries}
       {\thesubsection}
       {0.5em}
       {}
       [.]

\title{Connexive implications in Substructural Logics}

\usepackage[hidelinks]{hyperref}
\usepackage{aliascnt}
\usepackage{amsmath}
\usepackage[capitalise]{cleveref}
\usepackage{authblk}
\usepackage{underscore}
\usepackage{paralist}
\usepackage{bussproofs}
\usepackage{arydshln}
\usepackage{amsthm}

\usepackage{amssymb}
\usepackage{amsfonts}
\usepackage{mathtools}
\usepackage{tikz}
\usepackage{stmaryrd}
\usepackage{MnSymbol}
\usepackage{marginnote}
\usepackage{enumerate}
\usepackage{float}
\restylefloat{table}

\numberwithin{equation}{section}

\theoremstyle{plain}

\newtheorem{theorem}{Theorem}[section]
\newtheorem{lemma}[theorem]{Lemma}
\newtheorem{proposition}[theorem]{Proposition}
\newtheorem{corollary}[theorem]{Corollary}

\theoremstyle{definition}
\newtheorem{definition}[theorem]{Definition}

\newtheorem{remark}[theorem]{Remark}

\newtheorem{example}[theorem]{Example}

\crefname{proposition}{Prop.}{Props.}
\crefname{lemma}{Lem.}{Lems.}
\crefname{theorem}{Thm.}{Thms.}
\crefname{figure}{Fig.}{Figs.}
\crefname{corollary}{Cor.}{Cors.}
\crefname{example}{Ex.}{Exs.}
\crefname{remark}{Rm.}{Rms.}
\Crefname{proposition}{Proposition}{Propositions}
\Crefname{lemma}{Lemma}{Lemmas}
\Crefname{theorem}{Theorem}{Theorems}
\Crefname{corollary}{Corollary}{Corollaries}
\Crefname{figure}{Figure}{Figures}
\Crefname{example}{Example}{Examples}
\Crefname{remark}{Remark}{Remarks}

\newcommand{\m}{\mathbf}

\newcommand{\eq}{\approx}

\newcommand{\FLe}{{\mathsf{FL}_\mathsf{e}}}
\newcommand{\FLei}{\mathsf{FL}_\mathsf{ei}}
\newcommand{\FLew}{\mathsf{FL}_\mathsf{ew}}

\newcommand{\CRL}{\mathsf{CRL}}

\newcommand{\V}{\mathsf{V}}

\newcommand{\meet}{\wedge}
\newcommand{\join}{\vee}

\newcommand{\under}{{\to}}
\newcommand{\cnx}{{\Rightarrow}}
\newcommand{\rneg}{{\sim}}
\newcommand{\dm}{{\diamond}}

\newcommand{\bic}{\leftrightarrow}

\newcommand{\pos}{\delta}
\newcommand{\ppos}[1]{\pos{#1}}

\newcommand{\NS}{(\mathrm{NS})}

\newcommand{\BA}{\mathsf{BA}}
\newcommand\W{\mathsf{W}}
\newcommand\U{\mathsf{U}}

\newcommand{\IPL}{\mathbf{IPL}}
\newcommand{\CPL}{\mathbf{CPL}}

\newcommand{\K}{\mathsf{K}}

\newcommand{\Lng}{\mathcal{L}}

\newcommand{\cnxB}[1]{\cnx_{\hspace{-.2em}#1}}
\newcommand{\cnxBT}[2]{\cnxB{#1}^{\hspace{-.235em}#2}}

\newcommand{\cnxp}{\cnxB{\circ}}
\newcommand{\cnxm}{\cnxB{\meet}}

\newcommand{\cnxmt}{\cnxBT{\meet}{\pos}}
\newcommand{\cnxpt}{\cnxBT{\circ}{\pos}}
\newcommand{\cnxmd}{\cnxBT{{\meet}}{\dm}}
\newcommand{\cnxpd}{\cnxBT{\circ}{\dm}}

\newcommand{\spc}{strongly pseudo-complemented}
\newcommand{\expansive}{increasing}
\newcommand{\Spc}{\mathrm{(spc)}}

\newcommand{\glv}[1]{\dm{#1}}
\newcommand{\Glv}[1]{\glv{\m{#1}}}
\newcommand{\GLBA}[1]{\mathbf{G}_{#1}(\mathsf{BA})}
\DeclareMathOperator{\blto}{\blacktriangleright}

\newcommand{\Z}{\mathbb{Z}}

\newcommand{\FLeB}{\mathsf{FL}_\mathsf{e}^{\blacklozenge}}
\newcommand{\LgFLeB}{\mathbf{FL}_\mathbf{e}^{\blacklozenge}}
\newcommand{\FLeBA}{FL${}_{\mbox{\scriptsize e}}^\blacklozenge$}
\newcommand{\FLeA}{FL${}_{\mbox{\scriptsize e}}$}
\newcommand{\FLewA}{FL${}_{\mbox{\scriptsize ew}}$}

\newcommand{\Fm}{\mbox{\textit{Fm}}}

\newcommand{\Lg}{\mathbf{L}}

\usepackage{proof}
\usetikzlibrary{shapes,snakes}
\author[1]{Davide Fazio}
\author[2]{Gavin St. John}
\affil[1]{\small{Facolt\`a di Scienze della Comunicazione, Universit\`a degli Studi di Teramo, Campus ``Aurelio Saliceti'', Via R. Balzarini, 1, 64100, Teramo (TE), Italy}}
\affil[2]{Dipartimento di Matematica,
Universit\`a degli Studi di Salerno,
Via Giovanni Paolo II, 132,  
84084, Fisciano (SA), Italy.}
\affil[1]{\href{mailto:dfazio2@unite.it}{dfazio2@unite.it}}
\affil[2]{\href{mailto:gavinstjohn@gmail.com}{gavinstjohn@gmail.com}}
\date{}
\begin{document}
\maketitle
\begin{abstract}
This paper is devoted to the investigation of term-definable connexive implications in substructural logics with exchange and, on the semantical perspective, in sub-varieties of commutative residuated lattices (\FLeA-algebras). In particular, we inquire into sufficient and necessary conditions under which generalizations of the connexive implication-like operation defined in \cite{Falepa} for Heyting algebras still satisfy connexive theses. It will turn out that, in most cases, connexive principles are equivalent to the equational Glivenko property with respect to Boolean algebras. Furthermore, we provide some philosophical upshots like e.g., a discussion on the relevance of the above operation in relationship with G. Polya's logic of plausible inference, and some characterization results on weak and strong  connexity.
\end{abstract}
\section{Introduction}
The basic ideas of connexive logic can be traced back to Aristotle's \emph{Prior Analytics} and to Boethius' \emph{De hypotheticis syllogismis} (see \cite{mccall2012}). Connexive logic developed as an underground stream along the whole history of logic until the sixties of the last century. Since then, it has become a well established subject of investigation in non-classical logic. Connexive principles reflect a \emph{connection}, or \emph{compatibility}, between the antecedent and consequent of sound conditionals. Specifically, they establish  that a conditional statement ``if $A$, then $B$'' is sound provided that the negation of $B$ is incompatible with $A$. Such a connection can be expressed in a language containing a unary (negation) connective $\neg$ and a binary (implication) connective $\under$ by means of axioms in Table \ref{CnxPrinc}.
\begin{table}[ht]
\centering
\begin{tabular}{c | c c}
	Aristotle's Theses 
	& \quad&
		\begin{tabular}{c}
		$\neg(A\under\neg A)$\\
		$\neg(\neg A\under A)$
		\end{tabular}
	\\
	\hline
	Boethius' Theses
	&\quad &
		\begin{tabular}{c}
		$(A\under B)\under\neg(A\under\neg B)$ \\
		$(A\under \neg B)\under\neg(A\under B)$
		\end{tabular}
\end{tabular}
\caption{Basic connexive laws}\label{CnxPrinc}
\end{table}

\noindent A connexive logic is nothing but a logic having the above formulas as theorems with respect to a negation $\neg$ and a non-symmetric implication $\under $. The latter requirement, which we call \emph{the principle of non-symmetry}, is essential, since $\under$ must be understood as a genuine implication rather than as an equivalence. Apparently, the above formulas are falsified in classical logic whenever implications with \emph{false} antecedents are considered. Therefore, classical logic is not connexive with material implication and negation. On the other hand, while material biconditional does satisfy the theses in Table~\ref{CnxPrinc}, it is obviously symmetric.

Connexive theses have been motivated by different considerations depending on the specific meaning given to implication. One motivation comes from relevance logic and the idea that semantic consequence is a content relationship (see e.g. \cite{Ro78}). \cite{cant2008} argues that connexive systems can formalize indicative natural language conditionals. Some authors (e.g. \cite{kapom2017}) suggest that a connexive implication is suitable for modelling counterfactual conditionals (see also \cite{wansinghunter}).  Moreover,  \cite{mccall75} proposes an interpretation of connexive conditionals in terms of physical or ``causal'' implications. Also, the results of empirical research on the interpretation of negated conditionals (see \cite{mccall2012,Pf2012,Pf2017,wansingstanford}) suggest that speakers having no previous knowledge of formal logic are inclined to consider connexive conditionals as the sound ones. Recently, A. Kapsner has argued that the intuitive appeal of connexive principles is rooted, \emph{at least in the case of indicative conditionals}, in the \emph{use} of implication in concrete argumentation. In particular, it seems to depend on the assumption that a conditional statement of the form ``if $A$, then $B$'' is sound provided that $A$ is ``epistemically possible''\label{epistemically possible} for the speaker asserting it. Quoting Kapsner (\cite{Kaps2020}):
\begin{quote}
There is a presupposition for indicatives (which I called IP) that says that the antecedent needs to be epistemically possible for the speaker: If I say ``If today is Monday, Susy will come home tomorrow'', then this is only correct if I don't know for sure that today is not Monday. 
\end{quote}
The interested reader is referred to \cite{wansingstanford} for an exhaustive survey on these topics.

As stated in \cite[p. 381]{WanHom}, ``[...] the central concern of connexive logic consists of developing connexive systems that are naturally motivated conceptually or in terms of applications, that admit of a simple and plausible semantics, and that can be equipped with proof systems possessing nice proof-theoretical properties [...]''. 
Indeed, the literature on this subject offers a large amount of connexive logical systems ``built from scratch'' having interesting proof-theoretical features and, in some cases, transparent and elegant semantics. However, less has been said on the possibility of defining connexive implications within well established sub-logics of classical logic or some expansion thereof. In \cite{Pizzi91}, C.~Pizzi shows that a connexive arrow, called \emph{consequential implication}, can be defined by means of the modal notions of necessity and possibility within ordinary modal logic systems S1-S5, plus the system T. This approach has the advantage of suggesting a well studied framework like ordinary modal logic as a conceptual and formal basis for connexive semantics (cf. \cite{mccall2012}). In the same spirit, G.~Gherardi and E.~Orlandelli introduce super-strict implications \cite{GherOrla}. Very recently, D.~Fazio, A.~Ledda, and F.~Paoli \cite{Falepa} have investigated \emph{Connexive Heyting Logic} (CHL) which is algebraizable in the sense of Blok \& Pigozzi (see e.g., \cite{BlokPig,Font16}) with respect to a subvariety of H.\,P.\,Sankappanavar's semi-Heyting algebras satisfying an equational rendering of Aristotle's and Boethius' theses, i.e. \emph{Connexive Heyting algebras} (CHA). It turns out that CHA's are term-equivalent to Heyting algebras, and so CHL is \emph{deductively equivalent} to Intuitionistic Logic. More precisely, a CHA $\m A$ is a semi-Heyting algebra $(A,\land,\lor,\cnx,0,1)$ satisfying, among other identities, the following:\label{allcontheses}
\begin{align*}
1\approx (x\cnx y)\cnx\neg(x\cnx\neg y) &&1\approx (x\cnx \neg y)\cnx\neg(x\cnx y)\\
1\approx\neg(x\cnx\neg x) &&1\approx\neg(\neg x\cnx x)
\end{align*}
where $\neg x\coloneq  x\cnx 0$. Remarkably enough, $\cnx$ is, in general, not symmetric. Setting $x\under y\coloneq  x\cnx(x\land y)$, it follows by general results on semi-Heyting algebras that $\mathbb{H}(\m A)=(A,\land,\lor,\under,0,1)$ is a Heyting algebra. Conversely, given a Heyting algebra $\m A$, upon setting 
\begin{equation}
x\cnx y\coloneq  (x\under y)\land(y\to\neg\neg x) \label{heytcon}
\end{equation}
(where $\neg x$ is defined as expected), one has that $\mathbb{C}(\m A)=(A,\land,\lor,\cnx,0,1)$ is a CHA. Furthermore, $\mathbb{C}$ and $\mathbb{H}$ are mutually inverse mappings.

Due to the above features, CHL, which is the $1$-assertional logic of CHAs, can be entitled as a full-fledged connexive logic. Moreover, it enjoys properties which are indeed rare within connexive logics literature. For example, CHL is \emph{strongly connexive} in the sense of \cite{Kaps2012}. Furthermore, it allows one to investigate connexivity with well-known mathematical tools and with the conceptual ``arsenal'' provided by Intuitionistic Logic (think e.g. to the BHK semantics).

A somewhat suggestive interpretation of the above results might be the following. For any Heyting algebra $\m A$, the operation $\neg\neg\colon A\to A$ is a nucleus over the $\ell$-monoid $(A,\land,\lor,\cdot,1)$, where $\cdot\ \coloneq  \ \land$. Therefore, $\neg\neg$ can be regarded as a modal operator (cf. e.g. \cite{Young}). If we read $\neg\neg x$ as ``$x$ is not absolutely false'', or ``$x$ is plausible'', then a well behaving connexive arrow $\cnx$ can be obtained from intuitionistic implication $\rightarrow$ by strengthening the latter in such a way that no true (or at least plausible) statement can be implied by an \emph{a priori} false/implausible one. \label{interpretcnxheyting}Therefore, connexive Heyting implication might be seen as encoding, at a semantic level, a weak version of Kapsner's presupposition (IP) recalled above.

Given this observation, we frame this technique in a broader class of non-classical logics, {\em Substructural Logics}, which are understood as the external logics induced by axiomatic extensions of the {\em Full Lambek Calculus} $\mathtt{FL}$, a sequent system introduced to model natural language. In this manuscript we confine ourselves to the commutative setting, namely to $\mathtt{FL}_\mathtt{e}$, which is obtained from $\mathtt{FL}$ by adding the exchange structural rule $(\mathtt{e})$. Its extensions include many of the most well-studied non-classical logics: intuitionistic logic, relevance logic, linear logic, many-valued logics, with classical logic as a limit case. They find applications to areas as diverse as linguistics, philosophy, and theoretical computer science. Moreover, extensions of $\mathtt{FL}_\mathtt{e}$ turn out to be particularly appealing as they are algebraizable, in the sense of Blok \& Pigozzi, w.r.t. (pointed) commutative residuated lattices. 

A thorough treatment of substructural logics and residuated lattices can be found in \cite{yellowbook}.

In this paper, we work from the semantical perspective and investigate the class of pointed commutative residuated lattices (\FLeA-algebras) for which an implication-type connective $\cnx$ satisfying equational renderings of Boethius' and Aristotle's theses can be term-defined. Specifically, we define connectives similar to $\cnx$ in \eqref{heytcon} by considering the following candidates:\label{gener conneximpl}
\[x\cnxm y\coloneq  (x\under y)\land(y\under\neg\neg x)\quad\text{and}\quad x\cnxp y\coloneq  (x\under y)\cdot(y\under \neg\neg x),\] 
where $\neg x\coloneq  x\under 0$. 
This choice is motivated not only by the intrinsic interest for investigating mathematical properties of Heyting algebras underlying connexive principles, but also from more ``philosophical'' considerations. In fact, we will argue that the above interpretation of $\cnxm$ can be deepened in the broader framework of \FLeA-algebras and put into relationship with G.~Polya's theory of plausible inferences in mathematics developed in \cite{Polya}. Indeed, we will see that, under a reasonable notion of ``being more credible/plausible/likely to be true'', $\cnxm$ might formalize the kind of conditionals involving a conjecture and one of its consequences as the antecedent and consequent, respectively. More precisely, $\cnxm$ is the semantical counterpart of the \emph{weakest} term-definable (in the external logic of $\FLe$) implication-like connective $\leadsto$  satisfying axiomatic renderings of \emph{modus ponens} and Polya's \emph{fundamental inductive pattern}: from the truth of ``$A$ implies $B$'' and $B$, one  deduces that ``$A$ is more credible''.

This work aims at showing that, for an \FLeA-algebra $\m A$, a condition ensuring that $\cnxm$ and $\cnxp$ defined over $\m A$ satisfy Aristotle's and Boethius' Theses is that $\m A$ has the equational Glivenko property w.r.t. Boolean algebras (see below and e.g. \cite{yellowbook}). Such a requirement becomes also necessary when one deals with $\cnxm$. As a consequence, we obtain a novel characterization of \FLeA-algebras enjoying the equational Glivenko property relative to BA's by establishing a new link between connexivity and concepts of well-known mathematical depth.


Moreover, we turn our attention to \emph{integral} \FLeA-algebras. In this case we obtain somewhat surprising and strong results. Indeed, we consider the binary operations $\cnx$ in the (pointwise ordered) interval $[\cnxpt,\cnxmt]$ (where $\pos$ is an {\expansive} map replaying the role of $\neg\neg$) over an integral \FLeA-algebra $\m A$, and we show that they satisfy connexive axioms if and only if at least one of them does and $\pos=\neg\neg$. In this case any of the $\cnx$'s is symmetric if and only if $\m A$ is a Boolean algebra. Note that the latter result still holds, at least for $\cnxm$, if integrality is dropped.  Therefore, we conclude that, apart from Boolean algebras themselves, \FLeA-algebras with the equational Glivenko property w.r.t. BA's (and their $1$-assertional logics, see below) can be regarded as suitable environments in which connexive implications of a certain type can be defined.

Finally, we focus on the concepts of strong and \emph{weak} (see below and e.g. \cite{wansinghunter}) connexivity. In fact, on the one hand we will show that, in some cases, weak connexivity and connexivity can be regarded as one and the same thing. On the other, it will turn out that, although strong connexivity has often been associated to contra-classical (cf. \cite{Humberstone}). theses (see e.g. the notion os \emph{superconnexivity} in \cite{Kaps2012}), in the framework of extensions of $\mathbf{FL}_{\mathbf{e}}$ for which $\cnxm$ and $\cnxp$ are connexive, it is nevertheless equivalent to a rather ``classical'' inference schema: \emph{ex falso quodlibet}. Therefore, also in view of characterization theorems outlined above, this work yields an overall picture which is somewhat surprising: although connexive theses (and weak/strong connexivity) make connexive logics incomparable with classical logic, in the context we deal with, they mirror some distinguishing traits of intuitionistic logic.\\ 

The results discussed above suggest that a systematic (programmatic) investigation of connexive implication connectives which are \emph{term-definable} within well known systems of non-classical logic might have interesting consequences. On the one hand, it would provide (eventually strongly) connexive logics with transparent semantics and implications with intuitive meanings. Therefore, such investigations would perhaps shed some light on the ``semantic source'' of connexive principles. On the other, this line of research would establish relationships between connexive systems and logics which, like substructural logics, have relevant applications in logico-philosophical investigations as well as in mathematics, computer science, and Artificial Intelligence. This work aims at being a small step in that direction. 

Let us summarize the discourse of the paper. In Section \ref{sec: prelim} we dispatch the basic notions needed for the development of our arguments, and we provide some preliminary results. Section \ref{sec: connflealg} is devoted to alternative characterizations of those \FLeA-algebras in which $\cnxp$, $\cnxm$ and (in the integral case) any binary operation $\cnx$ ``in the between'', satisfy connexive theses. We prove that, for $\cnxm$ (and under a slightly stronger assumption $\cnxp$), the \emph{largest} class of algebras of this sort coincides with the variety of \FLeA-algebras with the equational Glivenko property w.r.t. Boolean algebras. In Section \ref{sec: philosophicalimplications}, building on results obtained in previous sections, we provide some philosophical upshots like e.g. a discussion on the relevance of $\cnxm$, a characterization of substructural logics for which $\cnxm$ ($\cnxp$) is strongly connexive, and some remarks on the relationship between weak connexivity and connexivity.
We conclude in Section \ref{sec: conclusion}.


\section{Basic notions and preliminary results}\label{sec: prelim}
In this section we provide the necessary notions for the contents of this manuscript. We will assume the reader has a working understanding in basic concepts of universal algebra, referring them to \cite{Burris} for further edification. 
\subsection{Residuated structures}
Here we recall the definitions and properties of residuated lattices and FL-algebras, but only define their commutative versions as their full generality is not needed for the purpose of this work. For a thorough treatment of these structures and their connection to substructural logics, we refer the reader to \cite{yellowbook}.  

A {\em commutative residuated lattice} (CRL) is an algebra $\m{R} = \langle R, \meet ,\join, \cdot, \to,1 \rangle$ such that $\langle R, \meet,\join \rangle$ is a lattice, $\langle R, \cdot,1\rangle$ is a commutative monoid, and $\m{R}$ satisfies the {\em law of residuation}, i.e., for all $x,y,z\in R$,
\begin{equation}\tag{residuation}\label{Res}
x\cdot y \leq z \iff x\leq y\to z,
\end{equation}
where $\leq$ is the induced lattice order; i.e., $x\leq y$ iff $x\meet y = x$. 
The law of residuation can be written equationally, and therefore the class of commutative residuated lattices form a variety denoted by $\CRL$. 
We will often abbreviate $\cdot$ by concatenation, i.e. $xy\coloneq  x\cdot y$, with the convention that concatenation binds tighter than the remaining connectives in the signature. We also use a bi-implication abbreviation $x \bic y \coloneq   (x\to y)\meet (y\to x)$.  
A residuated lattice is called {\em integral} if the monoid unit is the greatest element, i.e., it satisfies the identity ($\mathsf{i}$): $x \meet 1 \eq x$. We note that, if there is a greatest element $\top$ (not necessarily $1$), then $x\leq y$ implies $x\under y \eq \top$. If there is a least element $\bot$, then there must be a greatest element, namely $\bot\under\bot$, and $\bot$ is \emph{absorbing}, i.e., $\bot\cdot x \eq \bot$.

Below we recall some basic properties of CRLs, which we will mostly utilize without reference throughout this article.

\begin{proposition}\label[proposition]{RLfacts}
The following identities hold in $\CRL$:
\begin{enumerate}
\item $x(y \join z) \eq xy \join xz$
\item $x\to(y\meet z) \eq (x\to y) \meet (x\to z)$
\item $(x \join y)\to z \eq (x\to z) \meet (y\to z)$
\item $1\to x \eq x$
\item $x\to (y\to z) \eq yx\to z$
\end{enumerate}
Consequently, $\cdot$ is order-preserving in both coordinates, while $\under$ is order-preserving (reversing) in its right (left) coordinate.
\end{proposition}

An {\em \FLeA-algebra} is simply a $0$-pointed CRL, i.e., a CRL whose signature is expanded by a constant $0$, and the variety of \FLeA-algebras is denoted by $\FLe$. An \FLeA-algebra is \emph{$0$-bounded} if $0$ is the least element, i.e., it satisfies ($\mathsf{o}$): $0\meet x \eq 0$. By ($\mathsf{w}$) we denote {\em weakening}, defined via $(\mathsf{w}) = (\mathsf{i}) + (\mathsf{o})$. We denote the variety of \FLeA-algebras satisfying integrality (weakening) by $\FLei$ ($\FLew$).

Given that $0$ is a constant in the signature of \FLeA-algebras, we define the unary operation $\neg x \coloneq   x\to 0$, and use the convention that $\neg$ binds tighter than any connective. 
An \FLeA-algebra is said to be {\em involutive} if it satisfies $x\eq \neg\neg x$. 
An \FLeA-algebra is called {\em pseudo-complemented} if it satisfies the identity 
\begin{equation}\label{pc}\tag{pc}
x\meet \neg x\leq 0.
\end{equation}

We say an \FLeA-algebra is {\em \spc} if the following identity holds:
\begin{equation}\label{spc}\tag{spc}
\neg x \meet \neg(x\under y) \leq 0
\end{equation}
Clearly $\Spc$ entails (\ref{pc}) as an instance, however the converse need not hold as is witnessed by Figure~\ref{fig:pc not spc} below.

\begin{figure}[h]
\centering
\begin{tikzpicture}[baseline=(current bounding box.center)] 
\node[label=right: {$1$}] (1) at (0,1) {$\bullet$};
\node[label=right: {$0=0^2$}] (1) at (0,0) {$\bullet$};
\node[label=right: {$\bot$}] (1) at (0,-1) {$\bullet$};
\draw (0,-1)--(0,0)--(0,1);
\end{tikzpicture}
\qquad
$
\begin{array}{c | c c c}
\under & 1 & 0 & \bot \\ \hline
1&1&0&\bot \\
0&1&1&\bot\\
\bot&1&1&1
\end{array}
$
\label{fig:pc not spc}
\caption{The Hasse diagram of a pseudo-complemented integral \FLeA-algebra which does not satisfy $\Spc$ since $\neg 0 \meet \neg(0\under \bot) = \neg 0 \meet \neg \bot =1 \meet 1 = 1\neq 0$.} 
\end{figure}

\begin{remark}\label[remark]{rem:spc-pc}
The converse does hold for $0$-bounded \FLeA-algebras, as being $0$-bounded entails $\neg x\leq x\to y$, and hence $\neg x \meet \neg(x\under y)\leq \neg x\meet \neg\neg x \leq 0$.
\end{remark}

It is easy to see that, for \FLeA-algebras, the operation $\neg$ is order-reversing and satisfies $\neg x\eq  \neg\neg\neg x$. Additionally, the map $x\mapsto \neg\neg x$ is a {\em nucleus}; i.e., a map $\gamma\colon G\to G$ on a partially-ordered groupoid that is a {\em closure operator} [namely, a map $\gamma$ that is {\em \expansive:} $x\leq \gamma(x)$; {\em monotone:} $x\leq y$ implies $\gamma(x)\leq \gamma(y)$; and {\em idempotent:} $\gamma\circ\gamma = \gamma$] which further satisfies the identity $\gamma(x)\cdot \gamma(y)\leq \gamma(xy)$ [or equivalently, $\gamma(\gamma(x)\cdot \gamma(y))\eq \gamma(xy)$]. It is well known that, for any \FLeA-algebra $\m A$, if $\gamma\colon A\to A$ is a nucleus then $\m A_\gamma\coloneq  \langle\gamma[A], \meet, \join_\gamma,\cdot_\gamma,  \to, \gamma(0),\gamma(1)\rangle$ is also an \FLeA-algebra, where $x\cdot_\gamma y\coloneq  \gamma(x\cdot y)$ and $x\vee_\gamma y\coloneq   \gamma( x\vee y)$ [E.g., see Chapter 3.4.11 in \cite{yellowbook}]. 

For simplicity, we use the notation $\dm x \coloneq   \neg\neg x$ henceforth, and summarize the basic properties of $\dm$, all of which hold for any nucleus, in the proposition below, which we may often utilize without reference.

\begin{proposition}\label[proposition]{dblr}
The following (quasi-) identities hold in $\FLe$:
\begin{enumerate}
\item $x\leq \dm x$
\item $x\leq y \implies \dm x \leq \dm y$ 
\item $\dm{\dm x} \eq \dm{x}$
\item $\dm x \cdot \dm y \leq \dm(xy)$
\item $\dm x \meet \dm y \eq \dm(\dm x \meet \dm y)$
\item $\dm x \to \dm y \eq x\to \dm y \eq \dm(x\to \dm y) $
\end{enumerate}
Consequently, for any $\m A\in \FLe$, $\Glv{A} \coloneq  \m A_\dm$ is an involutive \FLeA-algebra.
\end{proposition}

The following example will be a useful counter-model throughout this article.
\begin{example}\label{ex:Z(n)}
Let $\Z$ be the set of integers. It is well known that $\Z$ forms a (totally-ordered) commutative residuated lattice where $x\meet_\Z y \coloneq   \min\{x,y\}$; $x\join_\Z y\coloneq  \max\{x,y\}$, $x\cdot_\Z y\coloneq   x +y$; $x\under_\Z y: = y - x$; and $1_\Z\coloneq   0$. For fixed $n\in \Z$, by $\Z(n)$ we denote the \FLeA-algebra as an expansion of $\Z$ by taking $0_{\Z(n)}\coloneq  n$. In $\Z(n)$, it is clear that for all $x\in \Z$, $\neg x = n-x$ and $\neg\neg x = x$, so $\Z(n)$ is involutive, i.e., $\dm \Z(n)=\Z(n)$.
\end{example}

\subsection{Algebraization for substructural logics}
A \emph{logic} $\mathbf{L}$ over an algebraic language $\Lng$ is a \emph{structural consequence relation} (i.e., closed under substitutions and is reflexive, transitive, and monotone) ${\vdash}\subseteq\wp(\Fm_\Lng)\times \Fm_\Lng$. %
The interested reader is referred to \cite{Font16} for details.

The name ``\FLeA'' comes from the fact that \FLeA-algebras are the (unique) equivalent algebraic semantics of the \emph{external} logic $\mathbf{FL}_\mathbf{e}$ induced by the Full Lambek Calculus with exchange $\mathtt{FL}_{\mathtt{e}}$. 
Precisely, we have that, for any $\Phi\cup\psi\subseteq\Fm_\Lng$ of formulas over the language $\Lng=\{\meet,\join,\cdot,\under,0,1 \}$, \[\Phi\vdash_{\mathbf{FL}_\mathbf{e}}\psi\quad\text{iff}\quad\{\blto\varphi:\varphi\in\Phi\}\vdash_{\mathtt{FL}_{\mathtt{e}}}\blto\psi.\]
$\FLe$ is the equivalent algebraic semantics of $\mathbf{FL}_{\mathbf{e}}$ means that there exists a pair of ``mutually inverse'' mappings (called \emph{transformers}) $\tau:\Fm_{\Lng}\to \wp(\Fm_{\Lng}^{2})$ from formulas to sets of equations, and $\rho:\Fm_{\Lng}^{2}\to\wp(\Fm_{\Lng})$ from equations to sets of formulas such that, for any $\Phi\cup\psi\subseteq\Fm_\Lng$:
\begin{itemize}
\item $\Phi\vdash_{\mathbf{FL}_{\mathbf{e}}}\psi$ if and only if $\tau(\Phi)\models_{\FLe}\tau(\psi)$;
\item $x\approx y\leftmodels\models\tau\rho(x\approx y)$,
\end{itemize}
or, equivalently, for any set of equations $\{\epsilon_{i}\approx\delta_{i}:i\in I\}\cup\{\epsilon\approx\delta\}\subseteq Fm^{2}_{\Lng}$:
\begin{itemize}
\item $\{\epsilon_{i}\approx\delta_{i}:i\in I\}\models_{\FLe}\epsilon\approx\delta$ iff $\{\rho(\epsilon_{i}\approx\delta_{i}):i\in I\}\vdash_{\mathbf{FL}_{\mathbf{e}}}\rho(\epsilon\approx\delta)$;
\item $x\dashv\vdash_{\mathbf{FL}_{\mathbf{e}}}\rho(\tau(x))$. 
\end{itemize}
It can be seen that, in our framework, a suitable pair of transformers is given by $\tau(\varphi):=\{1\leq\varphi\}$, while $\rho(\epsilon\approx\delta):=\{(x\to y)\land(y\to x)\}$. Cf. \cite{GalOno} for details. For this reason, $\mathbf{FL}_{\mathbf{e}}$ is often referred to as the $1$-\emph{assertional} logic of $\FLe$. Now, the lattice of axiomatic extensions of $\mathbf{FL}_\mathbf{e}$, in this paper called \emph{substructural logics}\footnote{Typically, ``substructural logic'' refers to any axiomatic extension of the logic $\mathbf{FL}$, i.e. the external consequence relation induced by the Full Lambek Calculus $\mathtt{FL}$ without the axiom of exchange (commutativity in the corresponding algebras), where the same algebraization results hold. See \cite{yellowbook}.}, is dually isomorphic to the lattice of subvarieties of $\FLe$. Indeed, for every class $\K$ of \FLeA-algebras and for every set $\Phi\subseteq\Fm_\Lng$ of formulas over $\Lng$, let $\Lg(\K):= \{\phi\in\Fm_\Lng: \K \models 1\leq \phi \}$ and $\V(\Phi):=\FLe\cap \mathrm{ Mod}(\{ 1\leq \phi: \phi\in \Phi\})$, we state the well-known algebraization theorem (e.g., Theorem 2.29 in \cite{yellowbook}):
\begin{theorem}[The Algebraization Theorem]\label[theorem]{FLeLogics}
The following hold:
\begin{enumerate}
\item For every substructural logic $\Lg$, $\Lg$ is algebraizable with $\V(\Lg)$ as its equivalent algebraic semantics. In particular, for any $\Phi\cup\{ \gamma\}\subseteq\Fm_\Lng$: \[\Phi \vdash_\Lg \gamma\quad\text{iff}\quad\{1\leq \varphi: \varphi\in \Phi \} \models_{\V(\Lg)} 1\leq \gamma;\]
\item For any subvariety $\mathsf{V}\subseteq\FLe$, $\mathsf{V}$ is the equivalent algebraic semantics of the logic $\Lg(\V)$. Specifically, for any $E\cup\{s\eq t\}\subseteq\wp({\Fm_{\Lng}}^{2})$, one has: 
\[E\models_\V s\eq t\quad\text{iff}\quad\{ u\to v, v\to u : (u\eq v)\in E\} \vdash_{\Lg(\V)} (s\under t) \meet (t\under s).\footnote{Cf. \cite[Definition 3.9]{Font16}.}\]
\end{enumerate}
\end{theorem}

\subsection{Boolean algebras and their Glivenko varieties}

It is well known that Boolean algebras are term-equivalent to \FLeA-algebras satisfying $x\cdot y \eq x\meet y$ and $x \under y \eq \neg x \vee y$ (e.g., Lem.~8.32 in \cite{yellowbook}). By $\mathsf{BA}$ we will denote the variety of \FLeA-algebras which are term-equivalent to Boolean algebras. It is well known that Boolean algebras are the equivalent algebraic semantics of classical logic $\CPL$.\footnote{I.e., $\BA=\V(\CPL)$ and $\CPL=\Lg(\BA)$.} The following proposition is likely folklore, but we provide its proof as it will be useful later.

\begin{proposition}\label[proposition]{FLeBool}
Let $\m A$ be an \FLeA-algebra. Then $\m A$ is (term-equivalent to) a Boolean algebra iff $\m A$ is integral, involutive, and pseudo-complemented.
\end{proposition}
\begin{proof}
The forward direction is obvious. For the reverse direction, let $x,y\in L$. First, by using the properties of integrality, commutativity, and pseudo-complemented, we compute $\neg(xy)\cdot(x\meet y) \leq \neg(xy)x \meet y = (x\under \neg y)x \meet y \leq \neg y \meet y \leq 0,$ 
hence $x\meet y \leq \neg\neg(xy) = xy$ by residuation and involutivity. Since $xy\leq x\meet y $ by integrality, we obtain $x\cdot y = x\meet y$. Hence, using this fact and involutivity, we obtain
$$x\under y = x\under \neg\neg y = \neg(x\cdot\neg y) =\neg (x\meet \neg y) = \neg (\neg\neg x\meet \neg y)=\neg\neg(\neg x \vee \neg\neg y)=\neg x\vee y. $$
\end{proof}

It is well known that classical propositional logic $\CPL$ is {\em translatable} into intuitionistic propositional logic $\IPL$; namely, $\vdash_\CPL \varphi$ iff $\vdash_\IPL \neg\neg\varphi$, for any formula $\varphi$. %
This concept, known as the Glivenko property, has been studied for substructural logics in general (see \cite{GalOno,yellowbook}), as the property is reflected via their algebraic semantics via the algebraization theorem. Indeed, for substructural logics $\Lg$ and $\mathbf{K}$, we say the \emph{Glivenko property holds for $\mathbf{K}$ relative to $\Lg$} if for any formula $\varphi$, $\vdash_\Lg \varphi$ iff $\vdash_{\mathbf{K}}\neg\neg\varphi$. Indeed, by \Cref{FLeLogics}, this equivalently stated: for any \FLeA-term $t$, $\V(\Lg)\models 1\leq t$ iff $\V(\mathbf{K})\models 1\leq \neg\neg t$. 

\begin{proposition}[See \cite{yellowbook, GalOno}]
For subvarieties $\V,\W\subseteq\FLe$ and \FLeA-terms $s,t$, the following are equivalent:
\begin{enumerate}
\item $\V\models 1\leq t$ iff $\W\models 1\leq \neg\neg t$.
\item $\V\models s\leq t$ iff $\W\models \neg t\leq \neg s$.
\item $\V\models s\eq t$ iff $\W\models \neg s \eq \neg t$.
\end{enumerate}
\end{proposition}

Thus, for varieties $\V$ and $\W$ of \FLeA-algebras, we say the {\em equational Glivenko property holds for $\W$ relative to $\V$} iff
\begin{equation}\label{eq:GlvProp}
\V \models s\eq t \iff \W\models \neg s \eq \neg t 
\end{equation}
for any equation $s\eq t$ in the language of $\FLe$. In particular, it follows that the equational Glivenko property holds for the variety of Heyting algebras $\mathsf{HA}$ relative to $\BA$, as it is well known $\mathsf{HA}=\V(\IPL)$. Note that the algebras in $\mathsf{HA}$ are exactly those \FLewA-algebras in which $\cdot$ coincides with $\meet$.

By $\mathbf{G}_\FLe(\V)$ we denote the largest subvariety $\W$ for which the equational Glivenko property holds for $\W$ relative to $\V$. By $\mathbf{G}_\U(\V)$ we denote $\mathbf{G}_\FLe(\V) \cap \U$, where $\U$ is a subvariety of $\FLe$, which we call \emph{the  Glivenko variety of $\U$ relative to $\V$}. Dually, for a substructural logic $\Lg$, $\mathbf{G}_{\mathbf{FL}_\mathbf{e}}(\Lg)$ denotes the \emph{least} (w.r.t. set inclusion) substructural logic $\mathbf{K}$ for which the Glivenko property holds relative to $\Lg$. If $\mathbf{K}$ is a substructural logic, by $\mathbf{G}_\mathbf{K}(\Lg)$ we denote $\mathbf{G}_{\mathbf{FL}_\mathbf{e}}(\Lg) \vee \mathbf{K}$, i.e., the \emph{smallest} extension of $\mathbf{K}$ having the Glivenko property w.r.t. $\mathbf{L}$, which we call \emph{the Glivenko logic of $\mathbf{K}$ relative to $\Lg$}.


\begin{theorem}[See \cite{yellowbook, GalOno}]\label{GLBA}~ 
\begin{enumerate}
\item $\GLBA{\FLe}$ is axiomatized relative to $\FLe$ by the following identities:
	\begin{enumerate}
	\item $\neg(x\cdot y)\eq \neg(x\meet y)$
	\item $1\leq \dm(\dm x\under x)$ [or, alternatively, $\neg(x\under y) \eq \neg(\neg x \vee y)$]
	\end{enumerate}
\item $\GLBA{\FLei}$ is axiomatized relative to $\FLei$ by the following identities:
	\begin{enumerate}
	\item $x\meet \neg x \leq 0$ [or, alternatively, $\neg(x\meet y)\eq \neg(x\cdot y)$]
	\item $1\leq \dm(\dm x\under x)$
	\end{enumerate}
\end{enumerate}
\end{theorem}

\begin{remark}\label[remark]{GLBAisBigger}
 Of course, $\BA\subsetneq\GLBA{\FLew}\subseteq\GLBA{\FLei}\subseteq\GLBA{\FLe}$ as any Heyting algebra is in $\GLBA{\FLew}$ but there are Heyting algebras which are not Boolean, see e.g. \cite{yellowbook}.
\end{remark}

As they will be useful in what follows, we provide further characterizations for the Glivenko varieties of $\FLe$ ($\FLei$, and $\FLew$) relative to Boolean algebras.

\begin{lemma}\label[lemma]{GLBAbool}
For an \FLeA-algebra $\m A$, the following are equivalent:
\begin{enumerate}
\item $\m A\in \GLBA{\FLe}$.
\item $\m A$ is pseudo-complemented, satisfies $1\leq \dm(\dm x\under x)$, and has $\dm 1$ as its greatest element (i.e., $\dm \m A$ is integral).
\item $\dm \m A$ is Boolean and $\m A$ satisfies $1\leq \dm(\dm x\under x)$.
\end{enumerate}
\end{lemma}
\begin{proof}
(1) implies (2) is obvious. Since $\dm A$ is involutive by fiat, (2) implies (3) follows \Cref{FLeBool}. For (3) implies (1), suppose $\dm \m A$ is Boolean and $\m A$ satisfies $(\star):1\leq \dm(\dm x\under x)$. So $\dm 1$ is the greatest element of $\m A$ and the following identities hold in $\m A$:
$$\dm (x\cdot y)\eq \dm x\meet\dm y \quad\mathsf{and}\quad x\under \dm y \eq \dm(\neg x \vee y) .$$

By \Cref{GLBA}, we need only verify that $\m A$ satisfies $\neg(x\cdot y)\eq \neg(x\meet y)$, or equivalently, $\dm(x\cdot y)\eq \dm(x\meet y)$. Towards this, we will first show $\m A$ satisfies $(\star\star):\dm(x\under y)\eq x\under \dm y$. Note that we need only verify $x\under \dm y \leq \dm(x\under y)$. Let $x,y\in A$. Since $(x\under \dm y)\cdot x \leq \dm y$, we have
$$
\dm y \under y \leq [(x\under \dm y)\cdot x]\under y = (x\under \dm y) \under (x\under y),
$$
from which it follows that:
$$
\begin{array}{r c l l}
\dm 1 &=& \dm[\dm y \under y] &\mbox{By $(\star)$} \\
&\leq & \dm[(x\under \dm y) \under (x\under y)] \\
&\leq & (x\under \dm y) \under \dm(x\under y). \\
\end{array}
$$
Hence $x\under \dm y \leq \dm(x\under y)$ by residuation.

Now, since $\dm(x\meet y)\leq \dm(\dm x \meet \dm y)=\dm(xy)$, we need only verify $\dm(xy)\leq \dm(x\meet y)$, or equivalently, $1\leq \dm(xy)\to \dm(x\meet y)$. First, we observe
$$\begin{array}{r c l l}
 \dm 1 \under  1 &\leq & (x\vee y)\under 1 &\mbox{Since $x,y\leq\dm 1$}\\
&=&[x\under 1]\meet [y\under 1]\\
&\leq &  [x\under(y\under y)] \meet [y\under (x\under x)]\\
&=& [xy\under y]\meet[xy\under x]\\
&=& xy \under (x\meet y).
\end{array}$$
Thus we obtain
$$\begin{array}[b]{r c l l} 
1 &\leq& \dm[\dm1\under 1] & \mbox{By $(\star)$}\\
& \leq  &\dm[xy \under (x\meet y)] &\mbox{By the above}\\
& =& xy \under \dm(x\meet y) &\mbox{By $(\star\star)$}\\
 &=& \dm(xy) \under \dm(x\meet y) .
\end{array}
\qedhere
$$
\end{proof}

\begin{lemma}\label[lemma]{lem:integralGLBA}
For $\m A$ an integral \FLeA-algebra, the following are equivalent:
\begin{enumerate}
\item $\m A \in \GLBA{\FLei}$.
\item $\m A$ is \spc.
\item $\m A \models \neg(x\under y)\eq \neg(\neg x \vee y)$.
\end{enumerate}
\end{lemma}
\begin{proof}
(1) implies (2) is clear since in $\GLBA{\FLei}$ we have
$$\neg x \meet \neg (x\under y) \eq \neg x \meet \neg(\neg x \vee y)\leq \neg x \meet \neg\neg x \leq 0.  $$
For (2) implies (3), suppose $\m A$ is \spc and let $x,y\in A$. Note that $\neg(\neg x \vee y) = \dm x \meet \neg y$. On the one hand we have
$$ 
\begin{array}{rcll}
(x\under y)\cdot(\dm x \meet \neg y) &\leq &[(x\under y)\cdot\dm x]\meet \neg y\\
&\leq& [(\dm x \under \dm y) \cdot \dm x]\meet \neg y & \mbox{Since $x\under y \leq x\under \dm y = \dm x\under \dm y$}\\
&\leq  & \dm y \meet \neg y \\
&\leq& 0 &\mbox{By (\ref{spc})}.
\end{array}
$$
Hence $\dm x \meet \neg y \leq \neg (x\under y)$ by residuation. On the other hand, we have
$$
\begin{array}{r c l l}
(\neg x \vee y)\cdot\neg(x\under y) &=& [\neg x \cdot \neg(x\under y)] \vee [y \cdot \neg(x\under y)]\\
&\leq &  [\neg x \meet \neg(x\under y)] \vee [y \cdot \neg y] & \mbox{By integrality}\\
&\leq & [\neg x \meet \neg(x\under y)]  \vee 0 \\
&=& 0 &\mbox{By (\ref{spc})}.
\end{array}
$$
Hence $\neg (x\under y)\leq \neg(\neg x \vee y)$, and we are done.

Lastly, for (3) implies (1), by \Cref{GLBA}(2), it is enough to show that $\neg(x\meet y)\eq \neg(x\cdot y)$ holds in $\m A$. Note that the $\leq$-direction holds by integrality. Let $x,y\in A$. Note that $x\meet y \leq \dm x \meet \dm y = \neg(\neg x \vee \neg y)$, and $\neg(\neg x \vee \neg y) = \neg (x\under \neg y)$ by assumption. Hence
$$ 
(x\meet y)\cdot \neg(x y) \leq \neg (x\under \neg y) \cdot \neg(xy) = \neg (x\under \neg y) \cdot (x\under \neg y) \leq 0.
$$
Hence $\neg(x\cdot y)\leq \neg (x\meet y)$ by residuation and we are done.
\end{proof}

\begin{corollary}\label[corollary]{GLBAint}
$\GLBA{\FLei} = \FLei + \mbox{\em(\ref{spc})}$ and $\GLBA{\FLew} = \FLew + \mbox{\em(\ref{pc})}$.
\end{corollary}

\begin{proof}
The first claim immediately follows from \Cref{lem:integralGLBA}, while the second claim follows from the additional observation made in \Cref{rem:spc-pc}, namely that $\FLe + (\mathsf{o})\models \mbox{\eqref{spc} $\mathsf{iff}$ \eqref{pc}}$. 
\end{proof}

\section{Connexive \FLeA-algebras}\label{sec: connflealg}

In this section, we lay the groundwork for studying connexivity in substructural logics through the algebraic lens provided by their algebraic semantics, i.e. \FLeA-algebras. We will first begin by considering the case of when the inherent implication $\under$ is connexive, and after a short discussion, devote our time to studying the operations $\cnxm$ and $\cnxp$ (noted in the introduction), and their generalizations. We will show that varieties of \FLeA-algebras in which these operations satisfy the (corresponding algebraic notions of) connexive laws are fundamentally linked to Glivenko varieties relative to Boolean algebras -- and therefore their corresponding substructural logics having the Glivenko property w.r.t. Classical Logic. In doing so, we generalize a number of results from \cite{Falepa}, for which much of this investigation was inspired. Furthermore, we provide an entire interval of logics that are \emph{bonafide} examples of connexive logics, which moreover are equivalently axiomatized by just any one of the connexive theses of Boethius (or even Aristotle in the case of weakening).

\subsection{Connexive principles in \FLeA-algebras}
In light of the Algebraization Theorem for substructural logics [cf., \cref{FLeLogics}], the algebraic analogue for the connexive laws become evident. 
Let $\cnx$ and $\rneg$ be some binary and unary connectives, respectively, definable in the language $\{\meet,\join,\cdot,\under,0,1 \}$. 
Then a substructural logic $\Lg$ has a connexive law (i.e., one of Aristotle's or Boethius' theses from \Cref{CnxPrinc}) as a theorem, in terms of $\cnx$ and $\rneg$, if and only if the corresponding identity in \Cref{AlgCnxLaws} is modeled by subvariety $\V(\Lg)$ of \FLeA-algebras.
\begin{table}[ht]
\centering
\begin{tabular}{c | l}
	\begin{tabular}{c}
	Equational \\ Aristotle's Theses
	\end{tabular} 
	& 
	\parbox{0.6\linewidth}{
		\begin{align*}
		1 \leq  \rneg(x \cnx \rneg x) \tag{AT}\label{AT}\\
		1 \leq  \rneg(\rneg x \cnx x) \tag{AT'}\label{AT'}
		\end{align*}		
		}
	\\
	\hline
	\begin{tabular}{c}
	Equational \\ Boethius' Theses
	\end{tabular} 
	&
	\parbox{0.6\linewidth}{
		\begin{align*}
		1 \leq (x \cnx y) \cnx \rneg(x \cnx \rneg y) \tag{BT}\label{BT}\\
		1 \leq (x \cnx \rneg y) \cnx \rneg(x \cnx  y) \tag{BT'}\label{BT'}
		\end{align*}		
		}
\end{tabular}
\caption{Order-algebraic connexive laws}\label{AlgCnxLaws}
\end{table}

Moreover, the connective $\cnx$ satisfies the principle of non-symmetry in $\Lg$ iff $\V(\Lg)$ has a member in which $\cnx$ is not a symmetric relation, i.e., there exists $\m A \in \V(\Lg)$ such that
\begin{equation}\tag{NS}\label{NS}
(\exists x,y\in A)[  x \cnx y \nleq y \cnx x \quad \mathsf{or}\quad y \cnx x \nleq x \cnx y]
\end{equation}

Dually, for a variety $\V\subseteq \FLe$, $\V$ satisfies an identity from \Cref{AlgCnxLaws} if and only if the corresponding connexive law is a theorem of $\Lg(\V)$. Similarly, if $\V$ has a member in which $\cnx$ is not symmetric, then the logic $\Lg(\V)$ satisfies the principle of non-symmetry for $\cnx$.

This leads us to the following definition, whose broader generality will be useful for what follows. Let $\m A = \langle A; \leq, \rneg, 1 \rangle$ be a pointed ordered-algebra with a unary operation, i.e., $\langle A, \leq\rangle$ is a preorder and $\langle A;\rneg,1\rangle$ is an algebra of type $(1,0)$. For a function $\cnx\colon A^2\to A$, we say $(\m A,\cnx)$ is {\em proto-connexive} if the identities in \Cref{AlgCnxLaws} hold in $\m A$, and we say $\m A$ is {\em connexive} if, furthermore, the condition \eqref{NS} holds in $\m A$, i.e., that $\cnx$ is not symmetric.

It is worth observing that our notion of proto-connexivity is very general. In fact, it does not mention that $\cnx$ and $\sim$ should be understood as a binary and a unary operation satisfying some \emph{minimal requirements} to be entitled as (the semantical counterpart of) an implication-like and a negation-like connective, respectively. Indeed, our notion is aimed at capturing the behavior of those algebras which may serve \emph{in principle} as the (possibly equivalent) algebraic semantics of a $1$-assertional logic which, satisfies under some respect, connexive theses. 

We are now ready to state the primary definitions for what follows. First, for an \FLeA-algebra $\m A$, by $\m A'$ we will denote the pointed ordered-algebra $\langle A;\leq, \neg,1\rangle$, where $\leq$ is lattice order inherent to $\m A$ and $\neg$ is defined as usual via $\neg x := x\to 0$. In this way, the we specialize the above concepts to the framework of substructural logics and \FLeA-algebras: 
\begin{definition}\label[definition]{def:protoconnexive}
Let $\m A \in \FLe$ and $\cnx\colon A^2\to A$ a binary function. We say $(\m A, \cnx)$ is \emph{proto-connexive} if $(\m A',\cnx)$ is proto-connexive, and similarly, $(\m A, \cnx)$ is \emph{connexive} if $(\m A',\cnx)$ is connexive. %
If $\cnx$ is term-definable in the language of \FLeA-algebras, then for a variety $\V\subseteq \FLe$ we say $(\V,\cnx)$ is \emph{proto-connexive} if for every member of $\m A\in \V$, $(\m A, \cnx)$ is proto-connexive, and we say $(\V,\cnx)$ is \emph{connexive} if furthermore it contains member $\m A$ such that $(\m A, \cnx)$ is connexive [i.e., $\V\nmodels x\cnx y\eq y\cnx x$]. If $\Lg$ is a substructural logic, we say $(\Lg,\cnx)$ is (\emph{proto-}) \emph{connexive} if $(\V(\Lg),\cnx)$ is.
\end{definition}
\noindent Whenever no danger of confusion will be impending, we will say that a binary operation $\cnx$ over $\m{X}$ [taken to be an algebra, variety, or logic], is (proto-) connexive if $(\m X,\cnx)$ is.

The next proposition introduces some identities that will be useful for the development of our arguments.
\begin{proposition}\label[proposition]{P1P2}
Let $\m A $ be an \FLeA-algebra and let $\cnx\colon A^2\to A$. If $(\m A,\cnx)$ satisfies the identities 
\begin{align*}
1 &\leq  {x \cnx \neg\neg x}\tag{P1}\label{P1}\\
\neg\neg(x \cnx y) &\eq \neg(x \cnx \neg y \tag{P2})\label{P2}
\end{align*}
then $(\m A,\cnx)$ is proto-connexive.
\end{proposition}
\begin{proof}
Note that \eqref{P2} is equivalent to $\neg(x \cnx y) \eq \neg\neg(x \cnx \neg y)$ in the context of \FLeA-algberas by  \Cref{dblr}. Using this observation, by \eqref{P1} and \Cref{dblr}, we obtain \eqref{AT} via
$1\leq x\cnx \neg\neg x \leq \neg\neg(x\cnx \neg\neg x) \eq \neg(x\cnx \neg x).$
Similarly, we obtain \eqref{AT'} via $1\leq \neg x \cnx \neg\neg(\neg x)\leq \neg\neg(\neg x\cnx \neg x)\eq \neg(\neg x\cnx x).$ \eqref{BT} is derived by a single application of \eqref{P1} followed by \eqref{P2}:
$$1\leq (x\cnx y)\cnx \neg\neg(x\cnx y) \eq (x\cnx y)\cnx \neg(x\cnx \neg y), $$
and \eqref{BT'} is calculated similarly using the equivalent reformation of \eqref{P2}.
\end{proof}
\subsection{Connexive implications in $\FLe$-algebras}\label[section]{sec:cnxResid} 
The most natural question to address first, in the context of substructural logics, is to ask what happens when we assume that the residual operation $\under$ satisfies the connexive laws. The ramifications of such an assumption are almost immediate.

\begin{proposition}\label{thm: connexiveimplinreslat}
The following are equivalent for $\m A=\langle A,\meet,\join,\cdot,\under,0,1\rangle\in \FLe$:
\begin{enumerate}
\item $(\m A,\under)$ is proto-connexive.
\item $(\m A,\under)\models\mbox{\eqref{AT}}$.
\item $0$ is the greatest element in $\m A$.
\end{enumerate}
\end{proposition}
\begin{proof}
(1) implies (2) holds by definition. Suppose (2), then $\m A\models x\under\neg x\leq 0$ by residuation. Now, it is easily verified that $\FLe\models x\leq (y\land 1)\under x$ using \cref{RLfacts}, from which it follows $0\leq (x\land 1)\under 0$ and $(x\land 1)\under 0\leq (x\land 1)\under((x\land 1)\under 0)\leq 0$. Hence $\m A\models \neg(x\land 1)\eq 0$. So for any $x\in A$, we find $x\leq \neg\neg x \leq \neg(\neg x \meet 1) \leq 0$ via \cref{dblr}, establishing (3). Lastly, suppose (3) holds, and note that this entails $\m A\models \neg x \eq 0$, as $0\leq x\to 0$ by residuation of $x\cdot 0\leq 0$, from which it immediately follows $(\m A,\under)\models \mbox{\eqref{P2}}$. Since $1\leq x\under \neg\neg x$ is satisfied in any \FLeA-algebra, the desired result follows from \Cref{P1P2}.
\end{proof}

It is worth noting that an \FLeA-algebra $\m A$ having $0$ as its largest element is equivalently stated via $\m A\models \neg x\eq \neg y$, as the forward direction is shown in the proof above and the reverse direction follows the valuation $x\mapsto \neg z$ and $y\mapsto 1$. Therefore, the variety of \FLeA-algebras satisfying $x\leq 0$ is equal to the variety $\mathbf{G}_{\FLe}(\V_\emptyset)$, the Glivenko variety of $\FLe$ relative to the trivial variety $\V_\emptyset$. In fact:
\begin{proposition}\label[proposition]{GLtriv}
The variety $(\mathbf{G}_{\FLe}(\V_\emptyset),\under)$ is connexive.
\end{proposition}
\begin{proof}
From \Cref{thm: connexiveimplinreslat} we have that $(\mathbf{G}_{\FLe}(\V_\emptyset),\under)$ is proto-connexive. Thus it suffices to show that it has a member in which $\under$ is not symmetric. Indeed, suppose $\m A\in \mathbf{G}_{\FLe}(\V_\emptyset)$ satisfies $x\to y \eq y\to x$. Then, in particular, $1\under 0 \leq 0\under 1$, and so by residuation $0=1\cdot 0 = 1\cdot(1\under 0)\leq 1$, i.e., $0=1$ and $\m A$ is a trivial algebra.  Now, since expanding any non-trivial integral commutative residuated lattice by a new constant $0$ such that $0:=1$ yields a non-trivial algebra in $\mathbf{G}_{\FLe}(\mathsf{V}_{\emptyset})$, it follows that $\mathbf{G}_{\FLe}(\mathsf{V}_{\emptyset}) \neq \mathsf{V}_{\emptyset}$ and therefore $(\mathbf{G}_{\FLe}(\V_\emptyset),\under)\models\mbox{\eqref{NS}}$. 
\end{proof}

However, the logic $\Lg=\Lg(\mathbf{G}_{\FLe}(\V_\emptyset))$ is far from behaving as a desirable connexive logic. In fact, although $\Lg$ has, for example, Aristotle's theses among its theorems, one has also that $\vdash_\Lg\varphi\to\neg\varphi$, for any formula $\varphi$, since the identity $1\leq x\under\neg x$ holds in $\mathbf{G}_{\FLe}(\mathsf{V}_{\emptyset})$. The latter fact implies that $\varphi\vdash_{\Lg}\neg\varphi$, for any formula $\varphi$. Therefore, $\neg$ does not satisfy the minimal requirement for a connective to be entitled as a negation (see e.g. \cite{ArAvZa, Marcos2005, WaOd2016}).

In the light of the above discussion, it seems reasonable to wonder if there exists other \emph{term-definable} binary operations on \FLeA-algebras (or expansion thereof) behaving like a connexive implication with much more appealing features. To this aim, for the remainder of this paper, we will study the generalized definition of the connexive implication defined in \cite{Falepa} as follows. Given a map $\pos\colon A\to A$ on $\m A$, we define the two following operations:
\begin{align*}
x \cnxmt y &\coloneq   (x\under y) \meet (y\under \ppos{x})\\
x \cnxpt y &\coloneq   (x\under y) \cdot (y\under \ppos{x}) 
\end{align*}

We will pay special attention to the case where $\delta=\dm$, as the operations $\cnxpd$ and $\cnxmd$ are term-definable in the language of $\FLe$. Moreover, we will omit the superscript when no confusion can arise; i.e., $\cnxm\coloneq  \cnxmd$ and $\cnxp\coloneq  \cnxpd$. 

It is worth noticing that, for $\cnx\in \{\cnxp,\cnxm\}$, the identity $1\leq x\cnx x$ holds in all \FLeA-algebras as $1\leq x\under x\leq x\under \dm x$. Now, while $\cnx$ is neither order-preserving in its right argument nor order-reversing in its left argument, generally speaking, it does satisfy a weakened version of these properties:

\begin{proposition}\label[proposition]{CnxOrdProp}
The following quasi-identities hold in $\FLe$ for $\cnx\in \{\cnxp,\cnxm \}$:
\begin{align*}
x\leq y ~\mathsf{and}~ \dm x \eq \dm y &\quad\mathsf{implies}\quad z\cnx x \leq z\cnx y\\
x\leq y ~\mathsf{and}~ \dm x \eq \dm y &\quad\mathsf{implies}\quad y\cnx z \leq x\cnx z
\end{align*}
\end{proposition}
\begin{proof}
Let $\m A\in \FLe$ and let $x,y,z\in A$. Suppose $x\leq y$ and $\dm x = \dm y$. The former implies $y\under z\leq x\under z$ and $z\under x\leq z\under y$, by \cref{RLfacts}(2,3), while the latter implies  $x\under \dm z = y\under \dm z$, by \cref{dblr}(6). Since $*\in\{\cdot,\meet\}$ is order-preserving in both arguments, we obtain:
$$\begin{array}[b]{l}
z\cnx x
=(z\under x)*(x\under \dm z)=
(z\under x)*(y\under \dm z)\leq (z\under y)*(\dm y\under \dm z) = z\cnx y; \\
y\cnx z
=(y\under z)*(z\under \dm y)=(y\under z)*(z\under \dm x) \leq (x\under z)*(z\under \dm x) =x\cnx z.
\end{array} \qedhere$$
\end{proof}

In terms of their relationship to (equational) connexive principles, we now prove the following observation which will be important throughout this paper.

\begin{lemma}\label[lemma]{BTequiv0}
Let $\m A$ be an \FLeA-algebra and $\cnx\in\{ \cnxp,\cnxm\}$. Then $(\m A, \cnx)$ satisfies {\em(\ref{P1})}. Consequently, if $(\m A,\cnx)\models\mbox{\em(\ref{P2})}$ then $(\m A,\cnx)$ is proto-connexive.
\end{lemma}
\begin{proof} The first claim follows from \Cref{CnxOrdProp}, setting $y,z:=x$ and in the first quasi-identity, as $x\leq \dm x$, $\dm x=\dm\dm x$, and $1\leq x\cnx x$ always hold. 
Hence $(\m A, \cnx)\models(\mbox{{\ref{P1}}})$. The second claim follows by \Cref{P1P2}.
\end{proof}

\noindent In fact, even more can be said 
for the connective $\cnxm$, as we see below.
\begin{lemma}\label[lemma]{charlanfcnxone}
For any \FLeA-algebra $\m A$, the following hold:
\begin{enumerate}
\item For all $x,y\in A$, $1\leq x\cnxm y$ iff $x\leq y$ and $\dm x = \dm y$.
\item $\m A \models \dm(x\cnxm \dm y) \eq x\cnxm \dm y \eq \dm x\cnxm \dm y$.
\end{enumerate}
\end{lemma}
\begin{proof}
For the first claim, towards the forward direction, $1\leq x\under y$ and $1\leq y\under \dm x$ immediately follows by assumption. Hence $x\leq y\leq \dm x$, and furthermore this give $\dm x \leq \dm y \leq \dm \dm x = \dm x$. The reverse direction follows from \cref{CnxOrdProp}, setting $z:= x$ in the first quasi-identity, since $1\leq x\cnx x$ always holds.
The second claim immediately follows by \cref{dblr}(5,6).
\end{proof}

\noindent Furthermore, we obtain a (stronger) converse of \Cref{BTequiv0} for $\cnxm$.
\begin{lemma}\label[lemma]{BTequiv1}
Let $\m A$ be an \FLeA-algebra. Then the following are equivalent:
\begin{enumerate}
\item $(\m A, \cnxm)$ is proto-connexive.
\item $(\m A,\cnxm)\models \mbox{\em(\ref{BT})}$. 
\item $(\m A,\cnxm)\models \mbox{\em(\ref{BT'})}$. 
\item $(\m A,\cnxm)\models \mbox{\em(\ref{P2})}$. 
\item $(\m A,\cnxm)$ models the following identity, 
\begin{equation}\label{P3}\tag{P3}
\neg(x\cnx y) \eq x\cnx \neg y.
\end{equation} 
\end{enumerate}
\end{lemma}

\begin{proof}
For simplicity, set $\cnx \coloneq \cnxm$. Clearly (1) implies both (2) and (3), by definition, (4) and (5) are equivalent by \cref{charlanfcnxone}(2), and (4) implies (1) by \Cref{BTequiv0}. It is sufficient then to show that (2) and (3) each imply (4). Indeed, if $(\m A,\cnxm)$ satisfies (\ref{BT}) then for any $x,y\in A$ we have $1\leq (x\cnx y) \cnx \neg(x\cnx \neg y) $. Hence by \Cref{charlanfcnxone}(1), $\neg(x\cnx y)\approx \dm(x\cnx\neg y)$. 
A similar calculation shows (3) implies (4). 
\end{proof}
\begin{remark}
It is worth noticing that \eqref{P3} is a well known and motivated falsification condition for (connexive) implication (see e.g. \cite{Wansing2005,wansingstanford}). For example, \cite{Wansing2007} motivates \eqref{P3} by the introduction of negated syntactic types in Categorial Grammar (see e.g. \cite{BU}) and the consequent need of providing falsity conditions for the functor-type forming directional implications. Wansing's connexive logic $\mathsf{C}$ is algebraizable w.r.t. a slight modification of the equivalent algebraic semantics of paraconsistent Nelson's logic $\mathsf{N4}$ satisfying, among other axioms, \eqref{P3} (see \cite{FaOd}).
\end{remark}
\begin{example} The above lemma does not generally hold in $\FLe$ for the connective $\cnxp$. Indeed, consider $\Z(n)$ from Example~\ref{ex:Z(n)}, for some arbitrary $n\in \Z$. It is easy to verify that $(\Z(n),\cnxp)$ satisfies \eqref{BT} since $x\cnxp y = 0 = 1_{\Z(n)}$ for all $x,y\in \Z$. However, taking $n\neq 0$, we see 
$\neg (x\cnxp \neg y)=n\neq 0 = x\cnxp y = \dm (x\cnxp y).$
Moreover, fixing $n<0$ and $y=x$ above we have $\neg(x \cnxp\neg x) = n \not\geq 0=1_{\Z(n)}$, so $(\Z(n),\cnxp)$ does not satisfy \eqref{AT}. We do, however, obtain a restricted converse with $\cnxp$ relative to \FLeA-algebras satisfying $x\leq \dm 1$ [see \cref{meetlikearrows}].
\end{example}

Lastly, we prove the following proposition, whose generality will be useful in the next section.
\begin{proposition}\label[proposition]{AllCnxInTriv}
Suppose $\m A$ is an \FLeA-algebra in which $0$ is the greatest element. Then for any map $\pos: A^2 \to A$, the connective $\cnxpt$ is proto-connexive on $\m A$. For $\pos = \dm$, then the connective $\cnxmd$ is proto-connexive in $\m A$.
\end{proposition}
\begin{proof}
The first claim follows from the fact that $\m A$ satisfies $\neg x \eq 0$, hence also $x\cdot 0 \eq 0$ and $x\under \neg y \eq 0$, from which it follows $x\cnxpt \neg y  \eq 0$. Consequently, Aristotle's and Boethius' theses hold for $\cnxpt$. For the second claim, $x\cnxmd \neg y \eq 0$ holds in $\m A$, so $\cnxmd$ is proto-connexive.
\end{proof}

\subsection{Characterizing connexivity for $\cnxm$ and $\cnxp$}\label{sec:charwhencnxmdiscon}
We will now investigate the varieties of \FLeA-algebras satisfying the connexive laws for both $\cnxm$ and $\cnxp$, and their relationship to the Glivenko variety relative to Boolean Algebras, which is, at least implicitly, hinted at in \cite{Falepa}. We provide a characterization for $\cnxm$ in the general case, and a restricted one for $\cnxp$. 
To that aim, let us start with the following technical lemmas.

\begin{lemma}\label[lemma]{WhenDMisInt}
Let $\m A$ be an \FLeA-algebra such that $\dm\m A$ is integral. Then:
\begin{enumerate}
\item $\m A$ satisfies $x\under y \eq x\cnxm (\dm x \meet y)$
\item $\m A$ satisfies $x\cnxp\neg y \leq x\cnxm \neg y$ and $\dm x \cnxp x \leq \dm(\dm x \cnxm x)$.
\item Let $\cnx\in \{\cnxp,\cnxm\}$. Then the following hold:
\subitem{(a)} If $(\m A,\cnx)\models \mbox{\eqref{AT}}$ then $\m A$ is pseudo-complemented.
\subitem{(b)} If $(\m A,\cnx)\models\mbox{\em(\ref{P2})}$ then $\m A$ satisfies $1\leq \dm(\dm x \under x)$.
\end{enumerate}
\end{lemma}
\begin{proof} Let $*\in \{\meet,\cdot \}$, and note that $\m A \models \dm x \eq \dm x * \dm 1$ since $\dm 1$ is the greatest element (for the $\meet$ case), and $\FLe\models \dm 1\cdot \dm x \eq \dm x$ (for the $\cdot$ case)\footnote{$\dm x = \dm x \cdot 1 \leq \dm x\cdot \dm 1 \leq \dm(\dm x \cdot \dm 1)=\dm(x\cdot 1)=\dm x$.}. Furthermore, note that $\dm 1$ being the greatest element implies $\dm(\dm x\cnxm x) \eq \dm(\dm x \under x)$.

For (1), note that $\dm \m A$ being integral entails $\dm 1 = x\under \dm y$ whenever $x\leq \dm y$. So for $x,y\in A$, we have
$$\begin{array}{r c l l}
x\cnxm(\dm x\meet y) &=& [x\under (\dm x \meet y)] \meet [(\dm x \meet y)\under \dm x]\\
&=& [x\under (\dm x \meet y)] \meet \dm 1 \\
&=& (x\under \dm x) \meet(x\under y) \\
&=& \dm 1 \meet(x\under y)  \\
&=& x\under y .
\end{array} $$

For (2), first observe that $\dm\m A$ being integral entails that $xy\leq \dm x\meet \dm y$. Hence, from this observation and \cref{dblr}(6), 
$$x\cnxp \neg y=(x\cnx \neg y)\cdot (\neg y\under \dm x)=(x\cnx \neg y)\meet (\neg y\under \dm x)=x\cnxm \neg y.$$
 For the other identity, observe that $\dm x\cnxp x \leq (\dm x \under x)\cdot \dm 1\leq \dm(\dm x \under x)=\dm(\dm x \cnxm x)$.


For (3), let $\cnx\in \{\cnxp,\cnxm\}$. Towards (a), set $c\coloneq  x\meet\neg x$. By order-preservation and residuation, it is easy to see that $c\leq \neg c$. %
Since $\dm \m A$ is integral, we have, on the one hand, $\dm 1 = c\under \neg c $, and on the other hand $\dm c = \dm 1 \under \dm c \leq \neg c \under \dm c $. Now, if $(\m A,\cnx)$ satisfies \eqref{AT}, then we have
$$
\begin{array}[b]{r c ll}
0
&\geq& c\cnx \neg c &\mbox{By \eqref{AT} and residuation}\\
&\geq& c\cnxp \neg c & \mbox{By (2)}\\
&=& (c\under \neg c) \cdot (\neg c \under \dm c)\\
&=& \neg c\under \dm c &\mbox{Since $\dm 1 =c\under \neg c $}\\
&\geq& \dm c \\
&\geq &c \coloneq   x\meet\neg x.
\end{array}
 $$ 

For (b), by \Cref{P1P2} we that $(\m A,\cnx)\models \mbox{(\ref{AT'})}$. So for $x\in A$, we have
$$\begin{array}[b]{rcll}
1
&\leq& \neg(\dm x\cnx \neg  x) &\mbox{By {(\ref{AT'})}}\\
&=&\dm(\dm x\cnx   x) &\mbox{By (\ref{P2})} \\
&\leq& \dm(\dm x \cnxm x) & \mbox{By (2)}\\
&=& \dm(\dm x \under x).
\end{array}
\qedhere$$
\end{proof}
\begin{example}
The converse of Lemma \ref{WhenDMisInt}(1) holds for ${\cnx}= \cnxm$. Indeed, it is easily checked that 
$\FLe\models x\cnxm (\dm x\meet y)\leq x\under \dm x,$ 
So, if $\m A$ satisfies $x\under y \eq x\cnxm(\dm x \meet y)$, by taking $x=1$ we obtain
$y  = 1\under y \leq 1 \under \dm 1 = \dm 1.$
\end{example}

\begin{lemma}\label[lemma]{FLeCnxInGLBA}
Let $\m A$ be an \FLeA-algebra. If $(\m A,\cnxm)\models  \mbox{\em(\ref{P2})}$ then $\m A \in \GLBA{\FLe}$.
\end{lemma}
\begin{proof} 
Note that $(\m A,\cnxm)$ satisfies \eqref{AT} by \Cref{BTequiv1}. In light of \Cref{GLBAbool}(2) and \Cref{WhenDMisInt}(3), it suffices to verify that $\dm\m A$ is integral. Indeed, for $x\in A$, 
$$ 
\begin{array}{r c l l}
x\under \dm1&\geq& (x\under \dm1) \meet \dm x \\
&=& (x\under \dm1) \meet (\dm 1 \under \dm x) \\
&=& x \cnxm \dm 1\\
&=&\neg (x\cnxm 0)& \mbox{By (\ref{P3}) since $\dm 1 = \neg 0$}\\
&\geq & \neg (x\under 0) & \mbox{By def. of $\cnxm$}\\
&\geq& x &\mbox{By \cref{dblr}(1)},
\end{array}
$$
and so $x\cdot x\leq \dm 1$ by residuation. It easily follows (e.g. by the substitution $x\mapsto y\join 1$) that $\dm 1$ is the greatest element of $\m A$, i.e. $\dm \m A$ is integral. 
\end{proof}

\begin{remark}
The lemma above does not generally hold for the connective $\cnxp$. Indeed, consider $\Z(n)$ from Example~\ref{ex:Z(n)} for $n=0$. Clearly $\dm \Z(0)$ is not integral, but it is easily checked that $(\Z(0),\cnxp)$ satisfies (\ref{P2}).
\end{remark}

\begin{lemma}\label[lemma]{GLBAisCnx}
If $\m A \in \GLBA{\FLe}$ then $(\m A,\cnxm)$ is proto-connexive.
\end{lemma}
\begin{proof}
Boolean algebras are involutive and have ${\meet} ={\cdot}$, so it follows that 
$$\GLBA{\FLe}\models \neg [x\cnx y] \eq \neg[(x\under y)\meet (y\under x)].$$
Furthermore, it is easily verified that Boolean algebras satisfy the identities: 
$$(x\under \neg y)\meet (\neg y \under \neg\neg x) \eq (x\under \neg y)\meet (\neg y \under x) \eq \neg[(x\under y)\meet (y\under x)],$$ and hence it follows that $\GLBA{\FLe}$ satisfies $\neg(x\cnx \neg y) \eq \dm (x\cnx y)$. 
So $(\m A,\cnxm)$ satisfies (\ref{P2}), and therefore is proto-connexive by \Cref{BTequiv1}. 
\end{proof}

We now obtain our promised characterization from \Cref{BTequiv1,GLBAisCnx,FLeCnxInGLBA}.
\begin{theorem}\label[theorem]{thm: charconn}
Let $\V$ be a variety of \FLeA-algebras. Then $(\V,\cnxm)$ is proto-connexive [or satisfies any of the identities \eqref{BT}, \eqref{BT'}, or \eqref{P2}] if and only if $\V \subseteq\GLBA{\FLe}$. 
Consequently, the largest variety of \FLeA-algebras for which $\cnxm$ is proto-connexive is exactly $\GLBA{\FLe}$.
\end{theorem}

We now remark upon the connexivity of other connectives in $\GLBA{\FLe}$. We start by proving the following technical lemma, whose generality will be useful in the sequel.

\begin{lemma}\label[lemma]{meetlikearrows}
Let $\m A$ be an \FLeA-algebra such that $\dm\m A$ is integral. Let $\cnx$ be a binary operation over $\m A$ such that $(\m A,\cnx)$ satisfies (i) \eqref{P1} and the identities (ii) $x\cnxp y \leq \dm(x\cnx  y)\leq x\cnxm \dm y$. 
Then the following are equivalent:

\begin{enumerate}
\item $(\m A,\cnx)$ is proto-connexive.
\item $(\m A,\cnx)$ models any one of \eqref{BT}, \eqref{BT'}, or \eqref{P2}.
\item $(\m A,\cnxm)$ is proto-connexive [equivalently, $\m A \in \GLBA{\FLe}$].
\end{enumerate}
In particular, the above holds for $\cnx:= \cnxp$.
\end{lemma}
\begin{proof}
Let us begin with (1) implies (2). Clearly, (1) implies both \eqref{BT} and \eqref{BT'} by definition. We show that either of these are sufficient to obtain \eqref{P2}. Let $x,y\in A$ and define $a:= x\cnx y$ and $b:=x\cnx \neg y$. If \eqref{BT} holds, then $1\leq a\cnx\neg b$, and we observe 
\begin{equation}\label{P2argument}
\begin{array}{r c l l}
1 
&\leq& a\cnx \neg b \\
&\leq& \dm[a\cnx \neg b] & \mbox{By (ii)} \\
&\leq &  \dm[a\cnxm \neg b] & \mbox{By assumption on $\cnx$} \\
&=&  a \cnxm \neg b & \mbox{By \cref{charlanfcnxone}(2)}\\
\end{array}
\end{equation}
Hence, by \Cref{charlanfcnxone}(1), $\neg\neg  a = \neg b$. So $(\m A, \cnx)\models\mbox{\eqref{P2}}$. The very same argument \eqref{P2argument} follows for the \eqref{BT'}-case by swapping the roles of $a$ and $b$. Hence we have concluded (1) implies (2).

For (2) implies (3), we have that, by the arguments above, $(\m A,\cnx)$ satisfies \eqref{P2} in any case.  So (3) follows from that fact that $(\m A,\cnx)$ is proto-connexive by (i) and \Cref{P1P2}, i.e., it satisfies \eqref{BT}, and so by the same argument \eqref{P2argument} above (by setting $a:= x\cnxm y$ and $b:=x\cnxm \neg y$) we conclude $(\m A,\cnxm)\models \mbox{\eqref{P2}}$, and the result follows from \Cref{BTequiv1}.

Now suppose (3) holds. On the one hand, by \Cref{BTequiv1} $(\m A,\cnxm)\models\mbox{(P3)}$, so $\dm(x\cnxm y)=x\cnxm \dm y$. On the other hand, $\m A \in \GLBA{\FLe}$ by \Cref{thm: charconn}, so $\dm(x\cnxp y)= \dm(x\cnxm y)$ by \Cref{GLBA}(1a). Thus $\dm(x\cnx y)=\dm(x\cnxm y)$ follows the assumption that $\dm(x\cnxp y)\leq \dm(x\cnx y)\leq x\cnxm \dm y$. Hence $(\m A,\cnx)\models\mbox{\eqref{P2}}$ follows from the fact $(\m A,\cnxm)$ does. Since $(\m A,\cnx)\models \mbox{(\ref{P1})}$ by assumption, $(\m A,\cnx)$ is proto-connexive by \Cref{P1P2}. This completes the equivalences.

Lastly, clearly the above holds for $\cnx:=\cnxp$ since (i) and (ii) follow from \Cref{BTequiv0} and \Cref{WhenDMisInt}(2).
\end{proof}

\begin{corollary}\label[corollary]{charconnProd} 
Let $\V$ be a subvariety of \FLeA-algebras. Then $(\V,\cnxm)$ is proto-connexive iff $(\V,\cnxp)$ is proto-connexive and $\V$ satisfies the identity $x\leq \dm 1$. 
\end{corollary}\label{charconnPage} 

\begin{example}
The converse to the above does not hold if the assumption of $x\leq \dm 1$ is dropped; e.g., $\Z(0)$ has no largest element, $(\Z(0),\cnxp)$ is proto-connexive, but $(\Z(0),\cnxm)\nmodels \mbox{\eqref{BT}}$ (take $x=0$ and $y=1$).
\end{example}

Having established the proto-connexitivity of $\cnxm$ and $\cnxp$ (and other related operations) in the Glivenko variety relative to Boolean algebras, we now investigate the principle of non-symmetry \eqref{NS}.
 
\begin{lemma}\label[lemma]{cor:FLeNS}
Let $\m A\in \GLBA{\FLe}$ and let $\cnx\in\{\cnxp,\cnxm\}$. Then $(\m A,\cnx)$ falsifes \eqref{NS}, i.e., $\cnx$ is symmetric on $\m A$, if and only if $\m A$ is a Boolean algebra.
\end{lemma}
\begin{proof}
Of course, if $\m A$ is a Boolean algebra, then $\cnxp ={\leftrightarrow}=\cnxm$ which is indeed symmetric. For the converse direction, first note that $\dm 1$ is the largest element of $\m A$. Let $*\in \{\meet,\cdot\}$ and $\cnx\in \{\cnxp,\cnxm\}$
$$\dm 1 = (1\to\dm 1)*(\dm 1 \to\dm 1)=1\cnx\dm 1 =\dm 1\cnx 1= (\dm 1\to 1)*(1\to\dm 1)=(\dm 1\to 1)*\dm 1.$$
On the one hand, if ${*}={\cdot}$, then $(\dm 1\to 1)*\dm 1\leq 1$, so $1=\dm 1$. On the other hand, for ${*}={\meet}$ we have $(\dm 1\to 1)*\dm 1 = \dm1 \to 1$, so one has $\dm 1\cdot\dm 1\leq 1$ and we compute
$x=x\cdot 1\leq x\cdot\dm 1\leq\dm 1\cdot\dm 1\leq 1$. So $1=\dm 1$ in either case, and consequently $1=x\under 1$.

Now, $x=(1\to x)*(x\to\dm 1)=1\cnx x=x\cnx 1=(x\to 1)*(1\to\dm x)=\dm x.$ 
The involutivity of $\neg$, together with $\m A\in\GLBA{\FLe}$, yield directly that $\m A$ is Boolean.
\end{proof}

Therefore, by \cref{thm: charconn}, \cref{charconnProd}, \Cref{GLBAisBigger}, and the lemma above:
\begin{corollary}\label[corollary]{connexcon}
Let $\V$ be any variety of \FLeA-algebras in the interval between $\mathsf{HA}$ and $\GLBA{\FLe}$. Then the connectives $\cnxm$ and $\cnxp$ are both connexive in $\V$.
\end{corollary}



\subsection{Generalizations in the integral case}

We now consider the special case when dealing with \emph{integral} (and $0$-\emph{bounded}) \FLeA-algebras, and show that the results from Section~\ref{sec:charwhencnxmdiscon} can be strengthened. In fact, Theorem \ref{FLeSPC} shows that, in the setting of $\FLei$ ($\FLew$), a binary operation in the pointwise ordered interval $[\cnxp,\cnxm]$ is proto-connexive if and only if \emph{any} operation in the interval is. Moreover, $\dm$ is the only {\expansive} mapping for which this holds. Furthermore, we demonstrate that satisfying any one of Aristotle's theses is sufficient to guarantee proto-connexity for subvarieties of $\FLew$.

Since integral \FLeA-algebras satisfy the identity $x\cdot y \leq x\meet y$, given such an algebra $\m A$ and map $\pos$, it follows $x\cnxpt y\leq x\cnxmt y$ for all $x,y\in A$. We define the (nonempty) interval of functions $[\cnxpt,\cnxmt]\subseteq A^2\times A$ via: 
$$f\in[\cnxpt,\cnxmt] \iff (\forall x,y\in A)[ x\cnxpt y \leq f(x,y) \leq x\cnxmt y ].$$

\begin{lemma}\label[lemma]{TauIsDm}
Let $\m A$ be an integral \FLeA-algebra, $\pos$ be an {\expansive} map on $\m A$, and $\cnx\in[\cnxpt,\cnxmt]$. 
\begin{enumerate}
\item For all $x\in A$, $1\cnx x = x$ and $x\cnx 1=\ppos{x}$. 
\item If $(\m A, \cnx)$ models either {\eqref{AT}} or \eqref{AT'}, then $\m A$ is pseudo-complemented. 
\item If $(\m A,\cnx)\models\mbox{\eqref{BT}}$ then $\neg x\eq x\cnx 0$, ${\pos} = {\dm}$, and $\m A \in \GLBA{\FLei}$.
\end{enumerate}
\end{lemma}
\begin{proof}
(1) easily follows from the following calculations, by integrality and virtue of the fact that $\cnx\in[\cnxpt,\cnxmt]$ and $\pos$ is increasing: For $x\in A$,
\begin{align*}
1\cnxp x = (1\under x)\cdot (x\under \ppos 1)=x\cdot 1=x=x\meet 1=(1\under x)\meet (x\under \ppos 1) = 1\cnxm x;\\
x\cnxp 1 = (x\under 1)\cdot( 1\under \ppos x) = 1\cdot \ppos{x} = \ppos{x}= 1\cdot \ppos{x}=(x\under 1)\meet( 1\under \ppos x) = x\cnxm 1.
\end{align*}

For (2), following essentially the same argument as \Cref{WhenDMisInt}(3a). Suppose $(\m A,\cnx)$ satisfies \eqref{AT}, and let $x\in A$. Set $c\coloneq  x\meet\neg x$ and recall that $c\leq \neg c$. 
Since $ \m A$ is integral, we have, on the one hand, $ 1 = c\under \neg c $, and on the other hand $\ppos{c}  \leq \neg c \under \ppos{c} $. Hence
$$c\leq \ppos{c} \leq \neg c \to \ppos{c} = (c\under \neg c)\cdot (\neg c \to \ppos{c} ) = c\cnxpt \neg c \leq c\cnx \neg c\leq 0, $$
where the last line follows from \eqref{AT} by residuation. So $\m A$ is pseudo-complemented. The same argument follows for \eqref{AT'} by considering $\neg\neg c \cnx \neg c$, as $c\leq \neg c$ implies $\dm c\leq \neg c$, so $c\leq \pos(\dm c)\leq  (\dm c \under \neg c)\cdot (\neg c \under \pos(\dm c)) =\dm c \cnxp \neg c \leq \dm c \cnx \neg c \leq 0 $.

For (3), suppose $(\m A,\cnx)\models \eqref{BT}$. Using \eqref{BT} and (1), we have
$$1 = (1\cnx x) \cnx \neg (1 \cnx \neg x) = x \cnx \neg\neg x \leq x\cnxmt \dm x =(x\to \dm x) \meet (\dm x \to \ppos{x}) =  \dm x \under \ppos{x},  $$
where the last equality follows from integrality. 
So $\dm x \leq \ppos{x}$ by residuation. Since $\m A$ is integral, $0\cdot \neg x\leq 0$ and thus $0\leq \dm x\leq \ppos{x}$, so it follows that $ 0\under \ppos{x} = 1$. Hence
$$\neg x = (x\under 0)\cdot (0\under \ppos{x})= x\cnxpt 0 \leq  x\cnx 0 \leq x\cnxmt 0 = (x\under 0)\meet (0\under \ppos{x}) = \neg x ,$$
completing the first claim. Towards the second claim, it suffices to show $\ppos{x}\leq \dm x$, or equivalently $1\leq \ppos{x}\to \dm x$. Indeed, we see
$$ 
\begin{array}[b]{ r c l l}
1 &=& (x\cnx 1) \cnx \neg (x \cnx \neg 1)  & \mbox{By \eqref{BT}}\\
&=& \ppos{x} \cnx \neg (x \cnx 0) & \mbox{By (1)}\\
&=& \ppos{x} \cnx \dm x & \mbox{By (b)}\\
&\leq& \ppos{x} \cnxmt \dm x & \mbox{Def. of $\cnx$}\\
&=& \ppos{x}\under \dm x &\mbox{By integrality since $\dm x \under \ppos{x}=1$}.
\end{array}
$$
Hence $\pos = \dm$, and so $\cnx\in [\cnxp,\cnxm]$. Since $x\cnxp y \leq x \cnx y$, $\cnx $ clearly satisfies $\eqref{P1}$ by \Cref{BTequiv0}. Therefore, $\cnx$ satisfies the assumptions of \Cref{meetlikearrows}, and hence $\m A \in \GLBA{\FLe}$. Since $\m A$ is integral by assumption, the claim follows.
\end{proof}
The lemma above allows us to conclude, in contrast to the general case in Section~\ref{sec:charwhencnxmdiscon}, the following stronger theorems in the integral setting.
\begin{theorem}
Let $\m A$ be an integral \FLeA-algebra, $\pos$ an {\expansive} mapping on $\m A$, and $\cnx\in[\cnxpt,\cnxmt]$. Then $(\m A,\cnx)$ is proto-connexive but it does not satisfy $\NS$ if and only if $\pos = \dm$ and $\m A$ is a Boolean algebra.
\end{theorem}
\begin{proof}
Suppose that $(\m A,\cnx)$ is proto-connexive but it does not satisfy $\NS$. So $(\m A,\cnx)$ satisfies \eqref{BT}, and hence $\pos=\dm$ and $\m A\in \GLBA{\FLei}$ by \Cref{TauIsDm}(3), in particular $\pos\m A = \dm\m A \in \BA$. Since $\NS$ is not satisfied and hence, for any $x\in A$, $x\cnx 1 = 1\cnx x$.  Thus $x = \ppos{x}=\dm x$ by \Cref{TauIsDm}(1), and we conlcude $\dm = \mathrm{id}_\m{A}$. So $\m A=\dm\m A$ is a Boolean algebra.  

The converse direction easily follows upon noticing that, if $\m A$ is a Boolean algebra and $\dm=\pos$, then ${\cnx} = {\bic}$ follows from involutivity and fact that $\cdot$ and $\meet$ coincide in Boolean algebras. Of course, material equivalence in Boolean algebras is proto-connexive but fails $\NS$ by straightforward calculations. 
\end{proof}

\begin{theorem}\label[theorem]{FLeSPC}
For $\m A$ an integral \FLeA-algebra, the following are equivalent:
\begin{enumerate}
\item $\m A \in \GLBA{\FLei}$ (i.e., $\m A$ is \spc).
\item For all $\cnx\in [\cnxp,\cnxm]$, $(\m A,\cnx)$ is proto-connexive.
\item There exists an {\expansive} $\pos$ and $\cnx\in [\cnxpt,\cnxmt]$ such that $(\m A,\cnx)\models \mbox{\eqref{BT}}$.
\end{enumerate}
\end{theorem}
\begin{proof}
For (1) implies (2), note that $\cnx\in [\cnxp,\cnxm]$ implies $x\cnxp y \leq \dm(x\cnx y)\leq \dm(x\cnxm y)$ and $\dm(x\cnxm y) = x\cnxm \dm y$ by the fact that $\cnxm$ satisfies \eqref{P3} by (1). Since $ \dm \m A$ is integral, by fiat, (2) follows from \Cref{meetlikearrows}. (2) implies (3) is obvious, and (3) implies (1) by \Cref{TauIsDm}(3).
\end{proof}

\noindent In light of \Cref{GLBAint} and \Cref{TauIsDm}(2), the theorem above implies:
\begin{corollary}\label[corollary]{thm: FLewPC}
Let $\m A \in \FLew$. Then the following are equivalent.
\begin{enumerate}
\item $\m A \in \GLBA{\FLew}$ (i.e., $\m A$ is pseudo-complemented).
\item For all $\cnx\in [\cnxp,\cnxm]$, $(\m A,\cnx)$ is proto-connexive.
\item There exists an {\expansive}  $\pos$ and $\cnx\in [\cnxpt,\cnxmt]$ such that $(\m A,\cnx)\models\mbox{\eqref{AT}}$.
\end{enumerate}
\end{corollary}
\begin{example}
Note that the direction (4) to (1) in \Cref{thm: FLewPC} does not generally hold if $0$-boundedness is dropped. Indeed, Figure~\ref{fig:pc not spc} is an integral \FLeA-algebra satisfying Aristotle's thesis but not Boethius' thesis (take $x = 0$ and $y = \bot$) for $\cnxm$.
\end{example}

%
%

\subsection{Returning to the logics: The main results}In what follows, we put in good use \Cref{FLeLogics} to set substructural logics mimicking a connexive implication expressed in terms of $\cnxm$ and $\cnxp$ ``on the map''. First, as a consequence of \Cref{thm: charconn}, \Cref{meetlikearrows}, and \Cref{connexcon}, we obtain the following theorem.
\begin{theorem}\label[theorem]{LogCharFLe}
Let $\Lg$ be any substructural logic in the interval between $\mathbf{FL}_\mathbf{e}$ and $\IPL$. Then the following are equivalent:
\begin{enumerate}
\item $(\Lg,\cnxm)$ is a connexive logic.
\item $(\Lg,\cnxm)$ is a proto-connexive logic.
\item Any one of Boethius' theses for $\cnxm$ are theorems of $\Lg$.
\item $\Lg$ is an axiomatic exentision of $\mathbf{G}_{\mathbf{FL}_\mathbf{e}}(\CPL)$.
\end{enumerate}
The same holds for the connective $\cnxp$ if it is further assumed that $\varphi \under \neg\neg 1$ are theorems of $\Lg$.
\end{theorem}

%
Furthermore, in the presence of weakening, by \Cref{thm: FLewPC} the above theorem specializes:
\begin{theorem}\label[theorem]{LogCharFLew}
Let $\Lg$ be any logic in the interval between $\mathbf{FL}_\mathbf{ew}$ and $\IPL$. Then the following are equivalent for $\cnx\in\{\cnxm,\cnxp \}$:
\begin{enumerate}
\item $(\Lg,\cnx)$ is a connexive logic.
\item Any one of Aristotle's theses for $\cnx$ are theorems of $\Lg$.
\item $\Lg$ is an axiomatic extension of $\mathbf{G}_{\mathbf{FL}_\mathbf{ew}}(\CPL)$.
\end{enumerate}
\end{theorem}

Interestingly enough, in the framework of $\mathbf{FL}_{\mathbf{ew}}$, the above result ``does justice'' to the idea that ``Aristotle’s Thesis is the cornerstone of the logics belonging to the family of so-called connexive logics'' (cf. \cite{Pizzi04}).\\

Since the variety $\mathsf{HA}$ of Heyting algebras is pseudo-complemented, we recover the following result from \cite{Falepa} as a particular instance of \Cref{LogCharFLew}.
\begin{corollary}
$(\mathsf{HA},\cnxm)$ is connexive, and therefore $(\IPL,\cnxm)$ is connexive.
\end{corollary}

As the connectives $\cdot$ and $\meet$ coincide in Heyting algebras, so too do the operations $\cnxm$ and $\cnxp$. We therefore conclude with the following corollary.

\begin{corollary}
The connectives $\cnxm$ and $\cnxp$ are connexive for every logic in the interval between $\mathbf{G}_{\mathbf{FL}_\mathbf{e}}(\CPL)$ and $\IPL$.
\end{corollary}

\section{Some philosophical considerations}\label{sec: philosophicalimplications}
As we have seen through the above sections, the implication $\cnxm$ plays a prominent role in our investigation. Indeed, its intrinsic interest (see p. \pageref{interpretcnxheyting}) depends not only on that it behaves as a connexive implication with a desirable formal behavior, but also on its particularly smooth interpretation. Section \ref{sec:Polya} is devoted to deepening such a reading, highlighting interesting links between $\cnxm$ and the theory of plausible reasoning. 
Subsequently, we will put into good use the machinery developed so far in order to highlight some interesting features of $\cnxm$. The last part of this section provides an investigation of two variants of connexivity, namely \emph{weak} and \emph{strong} connexivity. Interestingly enough, we will show that, when dealing with $\FLew$, these concepts can be regarded as one and the same thing. Furthermore, we will point out that,  for any substructural logic, the strong connexivity of $\cnxm$ and $\cnxp$ is formally embodied by a well known inference schema of intuitionistic logic: ex falso quodlibet.   

\subsection{A focus on $\cnxm$}\label{sec:Polya}
In the second volume of his famous \emph{Mathematics and plausible reasoning} \cite{Polya}, George Polya aims at formulating several \emph{patterns} of plausible reasoning explicitly. Among them, he investigates an inference schema ``which is of so general use that we could extract it from almost any example'' (\cite[vol.2, p.3]{Polya}). Let $A$ be some clearly formulated conjecture which is, at present, neither proved, nor refuted. Also, let $B$ be some consequence of $A$ which we have neither proved, nor refuted, as well. For example, if we set $A$ to be Goldbach's conjecture 
\begin{center}
\emph{every even natural number greater than 2 is the sum of two prime numbers}, 
\end{center}
$B$ might be that $198388=p_1 + p_2$ for suitable primes $p_1$ and $p_2$. Indeed, although we do not know whether $A$ or $B$ is true, it is unquestionable that \emph{$A$ implies $B$}. We verify $B$. If $B$ turns out to be false, then performing \emph{modus tollens} we can conclude that $A$ is false as well. Otherwise, if we recognize that $B$ is true (and this is the case), then, although we do not have a proof of $A$, we can nevertheless conclude that $A$ is \emph{more credible}. In other words, we have applied the following \emph{fundamental inductive pattern} (in brief ``inductive pattern''  \cite[vol.2, p.4]{Polya}) or \emph{heuristic syllogism}:
\begin{center}
\AxiomC{$A$ implies $B$}
\AxiomC{$B$}
\BinaryInfC{$A$ is more credible}
\DisplayProof
\end{center}
As stated in \cite[p.34]{Polya1949}, a reasonable \emph{logic of plausible inference} should a) be general enough to include the use of inductive (in a broad sense) reasoning in mathematics; b) include the heuristic syllogism among its inference rules; and c) be fully qualitative, in the sense that ``[...] it is not possible to give a numerical value to the degree of credence attached to any statement considered''. If one takes into account the possibility of developing this proposal in a \emph{monotonic} setting, then it is reasonable to interpret ``$x$ is more credible'' (or ``$x$ is likely to be true'') as a modal operator.  Therefore, it becomes worthy of attention considering expansions of $\mathbf{FL}_\mathbf{e}$ in which assertions on ``plausibility'',  as well as the inductive pattern, are amenable of a formal treatment. 

To this aim, and to motivate a formal account of the notion of plausibility, one might rely on the following assumptions:
\begin{itemize}
\item[$\blacklozenge$1] If $\neg A$ is more credible, then $A$ is not;
\item[$\blacklozenge$2] If $\neg A$ is false, then $A$ is more credible.
\item[$\blacklozenge$3] If $A$ is equivalent to $B$, then $A$ is more credible if and only if $B$ is.
\end{itemize}
Note that $\blacklozenge$1 encodes a weak version of Polya's principle that ``non-$A$ more credible'' is equivalent to ``$A$ less credible'' (see \cite[vol. 2, p.23]{Polya}).

Let us consider then the language $\Lng=\{\land,\lor,\cdot,\rightarrow,\blacklozenge,0,1\}$, where $\blacklozenge$ is a unary (modal) operator. We will denote the absolutely free algebra over $\Lng$ generated by an infinite countable set of variables by $\mathbf\Fm_{\Lng }$. The formula $\blacklozenge\varphi$ will be read as ``$\varphi$ is more credible/plausible/likely to be true''. 
In order to include $\blacklozenge$1-$\blacklozenge$3 in our formal system, we consider the expansion $\vdash_{\LgFLeB }\subseteq\wp(\Fm_{\Lng})\times \Fm_{\Lng}$ of $\vdash_{\mathbf{FL}_\mathbf{e}}$ by the following axioms and inference schemas concerning the behavior of $\blacklozenge$:
\begin{enumerate}[{$\blacklozenge$}1:]
\item $\vdash_{\LgFLeB } \blacklozenge\neg\varphi\to\neg\blacklozenge\varphi$;
\item  $\vdash_{\LgFLeB } \neg\neg\varphi\to\blacklozenge\varphi$;
\item $\varphi\leftrightarrow\psi\vdash_{\LgFLeB }\blacklozenge\varphi\leftrightarrow\blacklozenge\psi$.
\end{enumerate}
Due to $\blacklozenge$3, it follows from general facts concerning algebraizable logics (see e.g. \cite[Proposition 3.31]{Font16}) that $\vdash_{\LgFLeB}$ is algebraizable. Its equivalent algebraic semantics is the variety $\FLeB$  of \FLeBA-algebras whose members are structures of the form $(A,\land,\lor,\cdot,\to,\blacklozenge,0,1)$ where $(A,\land,\lor,\cdot,\to,0,1)$ is an \FLeA-algebra, and the following identities hold:
\begin{itemize}
\item[A1:]  $\blacklozenge\neg x\leq\neg\blacklozenge x$;
\item[A2:] $\neg\neg x\leq\blacklozenge x.$
\end{itemize}
As a consequence, we have the following
\begin{proposition}
Any \FLeBA-algebra satisfies $\neg\neg x \approx \blacklozenge x$.
\end{proposition}
\begin{proof}
 One has that $\neg x\geq \neg\blacklozenge x\geq \blacklozenge\neg x\geq \neg\neg\neg x= \neg x$. We conclude that $\neg x=\neg\blacklozenge x$. Consequently, we have $\neg\neg x=\neg\neg\blacklozenge x\geq\blacklozenge x$ and $\neg\neg x=\blacklozenge x$.
\end{proof}

Now, we observe that, in general, the following \emph{does not hold}:
\begin{equation}
\varphi\to\psi,\psi\vdash_{\LgFLeB }\blacklozenge\varphi, \label{failindpat}
\end{equation}
i.e., $\to$ does not satisfy the inductive pattern. However, 
it is easy to see that $\cnxm$ is \emph{the weakest} term-definable implication-like connective $\leadsto$ satisfying, for any formulas $\varphi,\psi\in \Fm_{\Lng}$, the following (stronger) axiomatic renderings of \emph{modus ponens}, and the heuristic syllogism
\begin{enumerate}[{S}1:]
\item $\vdash_{\LgFLeB }(\varphi\leadsto\psi)\to(\varphi\to\psi)$;
\item $\vdash_{\LgFLeB } (\varphi\leadsto\psi)\to(\psi\to\blacklozenge\varphi)$,
\end{enumerate}
where ``weakest'' here means that, for any term-definable binary connective $\leadsto$ satisfying S1-S2, we have that, for any $\varphi,\psi\in \Fm_{\Lng}$: $$\vdash_{\LgFLeB }(\varphi\leadsto\psi)\to(\varphi\cnxm\psi).$$

In light of the above discussion, it can be argued that, under Polya's \emph{desiderata} \cite{Polya1949}, $\cnxm$ might be considered as a reasonable candidate for formalizing the kind of conditionals which express a connection between a conjecture and one of its consequences in a monotonic framework. In fact, given that in any \FLeA-algebra $\m A$ it holds that
$$\begin{array}{r c l } 
x\cnxm y&=& (x\to y)\land (y\to\dm x)\\ 
&=& (x\to y)\land (\dm y\to\dm x)\\
&=& (x\to y)\land (\dm x\to\dm y)\land (\dm y\to\dm x)\\
&=& (x\to y)\land (\dm x\leftrightarrow\dm y),
\end{array}$$
$x\cnxm y$ can be read as ``$x$ implies $y$ and $x$ is plausible/more credible/likely to be true if and only if $y$ is'',  whenever ``it is plausible/more credible/likely to be true that $x$'' is meant to satisfy $\blacklozenge 1-\blacklozenge3$, and so it can be formalized by $\dm$. Note that, under such interpretation, the identity 
\begin{equation}
x\leq\dm \label{trueisplausible}1 
\end{equation}
seems to be a reasonable assumption once $1$ is meant as an absolutely true statement and so $\dm 1$ can be read as ``an absolutely true statement is plausible''. Therefore, if we confine ourselves to consider \FLeA-algebras satisfying \eqref{trueisplausible}, then \Cref{charconnProd} states a full equivalence between the connexivity of $\cnxm$ and $\cnxp$ (cf. p. \pageref{charconnPage}).\label{motivationdm1}

\subsection{On the strength of $\cnxm$}\label{sec: strength of cnxm}
Beside having particularly fair motivations, $\cnxm$ plays a special role for our discussion, since it enjoys interesting properties which are not shared by its product-variant $\cnxp$. Indeed, the connexivity of $\cnxm$ can be formulated in a stronger form.

\begin{proposition}\label[proposition]{lem: propcnxmd}Let $\m A$ be an \FLeA-algebra. Then $(\m A,\cnxm)$ satisfies \emph{(\ref{BT})} if and only if it satisfies: \[1\leq(x\cnx y)\cnx[(y\cnx z)\cnx\neg(x\cnx\neg z)].\label{BT*}\tag{BT*}\]
 \end{proposition}
 \begin{proof}
Set $\cnx \coloneq \cnxm$. Concerning the left-to-right direction, observe that \eqref{BT*} has the form $1\leq r\cnx(s\cnx t)$, for suitable terms $r,s,t$, which holds iff (i) $r\leq s\cnx t$ and (ii) $s\cnx t \leq \dm r$ by the definition of $\cnx$. Clearly (i) holds iff $rs\leq t$ and $rt\leq \dm s$ hold, which are verified by simply using \Cref{RLfacts}, the fact that $\FLe\models a\to \dm b \eq \neg b\to \neg a $, and an application of (\ref{P3}), by \Cref{BTequiv1}, via $x\cnx z \leq \neg(x\cnx \neg z)$. Moreover, since $(\m A,\cnx)\models\mbox{\eqref{BT}}$, we have $\m A\in \GLBA{\FLe}$ by \Cref{thm: charconn}. The reader can easily verify that $\BA\models x\leftrightarrow y \eq (x\leftrightarrow z) \meet (z\leftrightarrow y)$. Therefore, using the Glivenko property, \Cref{GLBAisCnx}, and the fact that $\dm$ is a nucleus, it follows that property (ii) holds as well. 
Conversely, setting $y\coloneq  x$ in \eqref{BT*}, we derive $x\cnxm x\leq (x\cnxm z)\cnxm\neg(x\cnxm\neg z)$. Thus \eqref{BT} holds since $1\leq x\cnxm x$.
\end{proof}
\noindent We show the above proposition does not hold for $\cnxp$ with the following example. 
\begin{example}
Consider 
$A=\{\bot,0,a,b,c,1 \}$ with the lattice order via $\bot<0<b,c<a<1$. Let $\cdot$ be the commutative operation with unit $1$ described via: For all $x\in A$, $\bot$ is absorbing; $x\cdot x=x$ for $x\neq 0$; $0\cdot x = \bot$ for $x\neq 1$; $a\cdot x = x$ for $x=b,c$; and $b\cdot c = \bot$. It is not difficult to show that $\cdot$ is associative and distributes over joins and therefore, by general results (since the lattice is complete), has a residual operation $\under$. Hence $\m A = \langle A, \meet,\join,\cdot,\under,0,1 \rangle$ is an integral \FLeA-algebra, which is furthermore verified to be a member of $\GLBA{\FLei}$. Therefore, by \Cref{charconnProd} one has that $\cnxp$ is proto-connexive. However, noting that $\neg c=c\under b=b$, $\neg b=b\under c = c$, and $0\under \bot = a$, the failure of \eqref{BT*} is computed:  
$(0\cnxp 1)\cnxp((1\cnxp b)\cnxp\neg(0\cnxp c))=0\cnxp(b \cnxp \neg\neg c)=0\cnxp(b\cnxp c)=0\cnxp \bot=a\ngeq 1.$
\end{example}

Furthermore, we see that $\cnxm$ has a privileged status among arrows of the form $\cnxmt$ which can be defined over an \FLeA-algebra by means of an increasing map $\delta$. In fact, we show below that it is the \emph{only} operation of this form capable of satisfying connexive theses in the Glivenko variety relative to Boolean algebras.

\begin{proposition}\label[proposition]{lem: poseqdm} 
Let $\m A \in \GLBA{\FLe}$ and $\pos$ be an {\expansive} map on $\m A$. If $(\m A,\cnxmt)\models \mbox{\eqref{BT}}$ then $\pos = \dm$.
\end{proposition}

\begin{proof}
By \eqref{BT} we have $1\leq(1\cnxmt x)\cnxmt\neg(1\cnxmt\neg x)$, and from the definition of $\cnxmt$, it follows that $1\leq\neg(1\cnxmt\neg x)\under\pos(1\cnxmt x)$. 
Expanding this, we obtain
\begin{align*} 
 1&\leq \neg[(1\under\neg x)\land(\neg x\under\ppos{1})]\under\pos[(1\under x)\land(x\under\ppos{1})]\\
 &=\neg[\neg x\land(\neg x\under\ppos{1})]\under\pos[x\land(x\under\ppos{1})]\tag{$\ast$} \label{astaux}
\end{align*}
Setting $x\coloneq  1$, \eqref{astaux} yields $1\leq\neg[\neg 1\land(\neg 1\under\ppos{1})]\under\pos[1\land(1\under\ppos{1})]=(0\under0)\under\ppos{1}$, hence $0\to0\leq \ppos{1}$. Since $\m A \in \GLBA{\FLe}$ and $0\under 0$ is its greatest element, it follows that $0\under 0=\ppos{1}$. 
So $\ppos{1}$ is the greatest element of $\m A$, and thus the identity $y\leq y\under\ppos{1}$, in particular, holds in $\m A$.
Using this fact, \eqref{astaux} simplifies to $1\leq\neg\neg x\under\ppos{x}$, i.e. $\dm x\leq\ppos{x}$. Moreover, using the fact that $\pos$ is increasing and $\ppos{1}$ is the greatest element, it is easily deduced that $\cnxmt$ also satisfies \eqref{AT} as a consequence instantiating of $y\coloneq  x$ in \eqref{BT}. As $\m A \in \GLBA{\FLe}$, from \eqref{AT} we conclude
$$\begin{array}[b]{r c l} 
1&\leq& \neg(x\cnxmt\neg x)\\
&=& \neg((x\under\neg x)\land(\neg x\under\ppos{x}))\\
&=&\neg(\neg x\land\ppos{x})\\
&=&\neg(\neg x\cdot\ppos{x})\\
&=&\ppos{x}\under\dm x.
\end{array}
\qedhere
$$
\end{proof}

\noindent As a consequence of the above proposition, we have the following
\begin{corollary}\label[corollary]{cor: poseqdm} 
Let $\m A$ be an \FLeA-algebra. Then $(\m A,\cnxm)$ is proto-connexive if and only if it is the unique binary operation of the form $\cnxmt$ to be so.
\end{corollary}
\noindent The example below shows neither \cref{lem: poseqdm} nor the corollary above hold for $\cnxp$.
\begin{example}
Consider the lattice $A = \{\bot < 1 < a < 0 \}$, where $\cdot$ is commutative and idempotent with unit $1$, described via $\bot$ being absorbing and $0 \cdot x = 0$ for any $x\neq \bot$. It is easily checked that $\cdot$ has a residual $\to$ and hence $\m A = \langle A,\meet,\vee, \cdot,\to, 0 ,1 \rangle$ is an \FLeA-algebra. By \Cref{AllCnxInTriv}, $\cnxpt$ is proto-connexive in $\m A$ for any map $\pos$, in particular for $\dm$. Moreover, specifying $\pos$ to be the map defined via $\ppos{x}$ is $a$ if $x=a$, and $0$ otherwise, we see that $\pos$ is increasing (and idempotent) which differs from $\dm$.
\end{example}

A somewhat suggestive interpretation of the above result comes next. If one wants to expand \FLeA-algebras with a modal operator $\delta$ standing for `` being plausible/more credible/likely to be true'', then a reasonable, minimal condition that $\delta$ should satisfy is being extensive, since any true statement should be \emph{a fortiori} plausible/more credible/likely to be true. Now, if one interprets $\cnxmt$ as the kind of implication which might appear in Polya's heuristic syllogism, then it seems reasonable to assume that it satisfies connexive theses. In fact, one might notice that any \emph{conjecture} $A$ \emph{whose truth or falsity is unknown} should not have, among its consequences, both a statement $B$ and its negation $\neg B$ on pain of being \emph{a priori} false, a contradiction. Moreover, it is arguable that conditionals involved in heuristic syllogisms have ``epistemically possible'' antecedents (cf. p. \pageref{epistemically possible} and \cite{Kaps2020}). But then, in the light of \Cref{lem: poseqdm}, $\delta=\dm$, namely ``being plausible/more credible/likely to be true'' must be \emph{perforce} expressed by double negation. Since inquiring into the consequences of the above considerations is beyond the scope of the present work, we postpone them to future investigations.

\subsection{On weak connexitivity}\label{sec: weak connex}
In \cite{wansinghunter}, a weaker notion of connexivity is formulated. A logic endowed  with a binary and a unary connective $\cnx$ and $\rneg$, respectively, is called \emph{weakly connexive} if it satisfies Aristotle's theses and the following two \emph{weak} versions of Boethius theses:
\begin{align*}
A\cnx B &\vdash \rneg(A\cnx \rneg B)\\
A\cnx \rneg B &\vdash \rneg(A\cnx B)
\end{align*}

Clearly, in the light of \Cref{FLeLogics}, the above inference schemas hold in a substructural logic w.r.t. (term defined) connectives $\cnx$ and $\rneg$ if and only if the following quasi-identities, which we call \emph{equational weak Boethius theses}, hold:
\begin{align*}
1\leq x\cnx y &\quad\mathsf{implies}\quad1\leq \neg (x\cnx \neg y) \tag{BTw}\label{BTw}\\
1\leq x\cnx \neg y &\quad\mathsf{implies}\quad1\leq \neg (x\cnx  y) \tag{BTw'}\label{BTw'}\\
\end{align*}

We will use the same naming convention for weak connexivity as in (and in the same spirit of) \Cref{def:protoconnexive} for logics and (classes of) algebras.

\begin{lemma}\label[lemma]{wkBTlem}
For an \FLeA-algebra $\m A$ and $\cnx\in \{\cnxp,\cnxm \}$, the following hold: 
\begin{enumerate}
\item $(\m A,\cnx) \models \eqref{BTw}$ iff $(\m A,\cnx)\models \eqref{BTw'}$. 
\item 
If $(\m A,\cnxm) \models \eqref{BT}$ [or \eqref{BT'}] then $(\m A,\cnxm)\models \eqref{BTw}$ and \eqref{BTw'}.
\item If $\dm \m A$ is integral, then $(\m A,\cnxp) \models \eqref{BT}$ [or \eqref{BT'}] implies $(\m A,\cnxp)\models \eqref{BTw}$ and \eqref{BTw'}.
\item If $(\m A,\cnx) \models \eqref{BTw}$ then $(\m A,\cnx)\models \eqref{AT}$ and \eqref{AT'}.
\end{enumerate}
Consequently, $\cnx$ is weakly connexive for a class $\K$ of $\FLe$-algebras if and only if at least one of the weak Boethius theses holds for $\cnx$ in $\K$.
\end{lemma}
\begin{proof} Fix $\cnx\in \{\cnxp,\cnxm \}$ and let $*\in\{\cdot,\meet\}$. First, note that, as a consequence of first quasi-identity in \cref{CnxOrdProp},  \FLeA-algebras satisfy the following identity, 
\begin{equation}\label{wkcnxcomp}
x\cnx y \leq x\cnx \neg\neg y
\end{equation}
Now, for (1), suppose \eqref{BTw} holds and let $1\leq x\cnx \neg y$. Then by \eqref{BTw}, $1\leq \neg(x\cnx \neg\neg y)$. But $\neg(x\cnx \neg\neg y) \leq \neg(x\cnx y)$ by Eq.~\eqref{wkcnxcomp}, so \eqref{BTw'} holds. Conversely, suppose \eqref{BTw'} holds. If $1\leq x\cnx y$, then again by Eq.~\eqref{wkcnxcomp}, so too $1\leq x\cnx \neg\neg y$. Hence by \eqref{BTw'}, $1\leq \neg (x\cnx \neg y)$.

For (2) and (3), it is sufficient to verify that $(\m A,\cnxm) \models \eqref{BT}$ implies $(\m A,\cnxm)\models \eqref{BTw}$, in light of (1), \cref{BTequiv1} [for the case of (2)], and \cref{meetlikearrows} [for the case of (3)]. Indeed, suppose $1\leq a:=x\cnx y$. Let $b:=\neg (x\cnx \neg y)$. Then by \eqref{BT}, using antitonicity of the left argument for $\under$ and \cref{dblr}, we find
$$1\leq a\cnx b = (a\under b)*(b\under \dm a)\leq b*(b\under \dm a)=\dm b * \dm(b\under \dm a)\leq \dm b \meet (b\under \dm a)\leq \dm b=b, $$
so \eqref{BTw} holds. This completes (2) and (3). 

Lastly, (4) is immediate by (1) and the fact $1\leq x\cnx x$ and $1\leq \neg x\cnx \neg x$.
\end{proof}

It naturally rises the question if a converse of \Cref{wkBTlem}(2) can be proven, namely if weak connexivity is equivalent to connexivity. Unfortunately, the next example shows that, even under the assumption of integrality, the former concept is properly weaker than the latter. 
\begin{example}\label{example:weak connexivity does not imply conn}
Consider an arbitrary Heyting algebra $\m A=(A,\land,\lor,\to,0,1)$ containing an element $a$ such that $b=\neg a\to a\neq a$, e.g. a three-elements chain with operations defined in the expected way. Let $\m A^{*}$ be the algebra obtained from $\m A$ by setting $0:=a$. It is easily seen that $\m A^{*}$ is still a pointed commutative residuated lattice with a bottom element $\bot$ satisfying $\land=\cdot$. Moreover, one has that $1\leq x\cnxm y$ implies $x\leq y$ and $\neg x=\neg y$. So $\neg(x\cnxm\neg y)=\neg((x\under\neg y)\land(\neg y\under\dm x))=\neg(\neg x\land\dm x)=\neg 0 = 1.$ Therefore, $\m A^{*}$ satisfies \eqref{BTw} and, by \Cref{wkBTlem}, also \eqref{BTw'}, i.e. $\m A^{*}$ is weakly connexive. However, one has also e.g. $\neg(0\cnxm\bot) = \neg ((0\under\bot)\land(\bot\under\dm 0))=b\neq a=((0\under 1)\land 1\cnx\dm 0)=0\cnxm 1=0\cnxm\neg\bot$. Therefore, since \eqref{P3} fails, by \Cref{BTequiv1} we have that $(\m A^{*},\cnxm)$ is not proto-connexive.
\end{example}

However, with the assumption of weakening, 
\Cref{wkBTlem} and \Cref{thm: FLewPC} yield the equivalence below. First, for a variety $\V$ of \FLeA-algebras and term-definable connective $\cnx$, let us denote by $Q_w(\V)$ the quasi-variety axiomatized relative to $\V$ by the equational Aristotle's and weak Boethius' theses.\\ 

\begin{theorem}\label[theorem]{thm:weaklyconnexiveinflew}
If $\K$ is a class of \FLewA-algebras and $\cnx\in \{\cnxp,\cnxm \}$, then $\cnx$ is weakly connexive for $\K$ if and only if $\cnx$ is proto-connexive for $\K$. Moreover, for any subvariety of $\V\subseteq \FLew$, $Q_w(\V)$ is a variety, namely the variety $\V+\eqref{AT}$.
\end{theorem}

Let $\cnx\in \{\cnxp,\cnxm \}$ and $\mathbf{C}_\cnx^{wk}$ be the extension of $\mathbf{FL}_\mathbf{ew}$ by the weak Boethius' theses for $\cnx$. By algebraization and the theorem above, this logic is deductively equivalent to the extension of $\mathbf{FL}_\mathbf{ew}$ by any weak Boethius' thesis for $\cnx$. Note that $\IPL$ satisfies the weak Boethius theses for $\cnx$ since $(\mathsf{HA},\cnx)\models\eqref{AT}$. By well known facts on algebraization, \Cref{thm:weaklyconnexiveinflew} yields the following 
\begin{corollary}
Let $\cnx\in \{\cnxp,\cnxm \}$. Then, for any substructural logic $\Lg$ in the interval between $\mathbf{FL}_\mathbf{ew}$ and $\IPL$, $(\Lg,\cnx)$ is weakly connexive if and only if $(\Lg,\cnx)$ is connexive. Moreover, every logic between $\mathbf{C}_\cnx^{wk}$ and $\IPL$ is connexive for $\cnx$.
\end{corollary}

\noindent Consequently, $\mathbf{C}_\cnx^{wk}$ is deductively equivalent to $\mathbf{G}_{\mathbf{FL}_\mathbf{ew}}(\CPL)$.

\subsection{On strong connexivity}
Lastly, we put our results in the context of what is often referred to in the literature  as strong connexivity \cite{Kaps2012, WanHom}.  A logic is called {\em strongly connexive} if it satisfies Aristotle's and Boethius' theses w.r.t. a non-symmetric implication and, moreover, satisfies the requirements:
\begin{enumerate}
\item[(K1)] In no model, $A\cnx \neg A$ is satisfiable (for any $A$), and in no model, $\neg A\cnx A$ is satisfiable (for any $A$);
\item[(K2)] In no model, $A\cnx B$ and $A\cnx \neg B$ are simultaneously satisfiable (for any $A$ and $B$).
\end{enumerate}
Accordingly, let us define a logic to be {\em strongly connexive} if it is connexive and satisfies the principles (K1) and (K2). In this way, we also refer to a connective $\cnx$ being {\em strongly connexive} for some logic.

Now, if $\m A$ is an \FLeA-algebra in which $0$ is the largest element, both $\cnxm$ and $\cnxp$ are proto-connexive for $\m A$ (see \Cref{AllCnxInTriv}). Furthermore, if $\m A$ is nontrivial, then (K1) and (K2) are refuted. 
By the remarks above, 
we obtain the following
\begin{proposition}
For the logics $\mathbf{G}_{\mathbf{FL}_\mathbf{e}}(\CPL)$ and $\mathbf{G}_{\mathbf{FL}_\mathbf{ei}}(\CPL)$, 
both $\cnxm$ and $\cnxp$ are connexive but not strongly connexive.
\end{proposition}
It is easily shown that if either (K1) or (K2) are refuted in an \FLeA-algebra $\m A$, for either $\cnxm$ or $\cnxp$, then $\m A\models 1\leq 0$.
\begin{proposition}\label[proposition]{rem: stronconnexoneleqzero}
Let $\m A$ be an \FLeA-algebra.
\begin{enumerate}
\item If $(\m A,\cnxm)$ is proto-connexive, then it refutes \emph{(K1)} or \emph{(K2)} iff $\m A\models x\leq 0$;
\item If $(\m A,\cnxp)$ is proto-connexive and it refutes \emph{(K1)} or \emph{(K2)}, then $\m A\models 1\leq 0$. 
\end{enumerate}
\end{proposition}
\begin{proof}
(1). Concerning the right-to-left direction, note that $\m A\models x\leq 0$ implies that, for any $y\in A$, one has $1\leq y\under\neg y$ and $1\leq\neg y\under\dm y$. So we conclude that $1\leq y\cnxm\neg y$ and so $\cnxm$ is not strongly connexive over $A$. Conversely, let us distinguish the following cases:

\begin{enumerate}[(i)]
\item There exists $x\in A$ such that $1\leq x\cnxm\neg x$. This means by \Cref{charlanfcnxone}(1), $\neg x=\neg\neg x$. Since $\m A\in\GLBA{\FLe}$, we have that $\dm\m A$ is trivial and so we conclude $x\leq 0$. 
\item If there exist $x,y\in A$ such that $1\leq x\cnxm y$ and $1\leq x\cnxm\neg y$, then one has, again by \Cref{charlanfcnxone}(1), $\neg x =\neg y$ and $\neg x =\neg\neg y$, i.e. $\neg y=\neg\neg y$. Reasoning as in (1), the desired conclusion follows. 
\end{enumerate}
As regards (2), let us distinguish the following cases
\begin{enumerate}[(i)]
\item There exists $x\in A$ such that $1\leq x\cnxp\neg x$. Then, of course, one has $1\leq x\cnx\neg x\leq 0$;
\item If there exist $x,y\in A$ such that $1\leq x\cnxp y$ and $1\leq x\cnxp\neg y$, then we have:
\begin{align*}
1\leq& (x\cnxp y)(x\cnxp\neg y)\\
=&(x\under y)(y\under\dm x)(x\under\neg y)(\neg y\under\dm x)\\
\leq& (x\under\neg x)(\neg x\under\dm x) = (x\cnxp\neg x).
\end{align*}
By the previous case, our claim follows. \qedhere
\end{enumerate}
\end{proof}
Even more, the above results show that, for $\cnxm$, the failure of strong connexivity is witnessed \emph{exactly} by those \FLeA-algebras having the equational Glivenko property with respect to the trivial variety (cf. \Cref{sec:cnxResid}).  As there is no non-trivial algebra in $\FLew$ in which $0=1$, \Cref{rem: stronconnexoneleqzero} yields

\begin{theorem}\label[theorem]{StrCnxFLew}
Let $\cnx\in\{\cnxp,\cnxm\}$. Then, for any substructural logic $\Lg$ in the interval between $\mathbf{FL}_\mathbf{ew}$ and $\IPL$, $(\Lg,\cnx)$ is strongly connexive if and only if $(\Lg,\cnx)$ is connexive.
Moreover, every logic between $\mathbf{G}_{\mathbf{FL}_\mathbf{ew}}(\CPL)$ and $\IPL$ 
is strongly connexive for $\cnx$.
\end{theorem}

In light of \Cref{WkCnxFLew} and the above \Cref{StrCnxFLew}, we are confronted with the surprising fact: 
\begin{theorem}\label[theorem]{StrCnxWk}
For any logic between $\mathbf{FL}_\mathbf{ew}$ and $\IPL$, 
the conditions of strongly connexive, connexive, and weakly connexive are equivalently satisfied for both $\cnxp$ and $\cnxm$.
\end{theorem}

As already argued in \cite[p. 4]{Kaps2012}, a reasonable task to pursue is expressing strong connexivity in the object language itself. To this aim, Kapsner introduces the notion of \emph{superconnexivity} as the formal alter-ego of strong connexivity, since it reflects the idea that violations of Aristotle's and Boethius' theses should be regarded as genuine contradictions whose satisfaction results into triviality. The latter fact can be codified by the following axiom schemas
\begin{itemize}
\item[SA:] $(\varphi\to\neg\varphi)\under\psi$;
\item[SB:] $(\varphi\to\psi)\to((\varphi\to\neg\psi)\to\chi)$. 
\end{itemize}
Unfortunately, SA and SB lead to triviality under a very narrow set of assumptions (among the others, being closed under substitution). Therefore, one might wonder whether, at least in some cases, strong connexivity is conveyed by less demanding inference rules. And the answer is positive. In fact, in what follows we show that, at least in our framework, strong connexivity can be expressed by means of a simple (indeed very classical!) schema: \emph{ex falso quodlibet}.\\
Let us denote by $\mathsf{PC}^{\circ}$ and $\mathsf{PC}^{\land} [=\GLBA{\FLe}]$ the varieties of \FLeA-algebras in which $\cnxp$ and $\cnxm$ are proto-connexive, respectively. In view of Proposition \ref{rem: stronconnexoneleqzero}, the \emph{largest} sub-quasivariety $\mathsf{M}^{\circ}$ ($\mathsf{M}^{\land}$) of $\mathsf{PC}^{\circ}$ ($\mathsf{PC}^{\land}$) in whose $1$-assertional logic $\cnxp$ ($\cnxm$) is strongly connexive is axiomatized by the quasi-identity \[1\leq 0\quad\mathsf{implies}\quad 1\leq x\] (or, equivalently, $1\leq 0\quad\mathsf{implies}\quad x=y$). Therefore, we have that $\Lg(\mathsf{M}^{\circ})$ ($\Lg(\mathsf{M}^{\land})$) can be obtained from $\vdash_{\Lg(\mathsf{PC}^{\circ})}$ ($\vdash_{\Lg(\mathsf{PC}^{\land})}$) by adding the \emph{explosion} schema 
\begin{equation}
0\vdash\varphi.\tag{EFQ}\label{efq} 
\end{equation}
Of course, $\Lg(\mathsf{M}^{\circ})$ ($\Lg(\mathsf{M}^{\land})$) is the \emph{least} $1$-assertional sub-logic of ${\Lg(\mathsf{PC}^{\circ})}$ (${\Lg(\mathsf{PC}^{\land})}$) for which $\cnxp$ ($\cnxm$) is strongly connexive. The above considerations boil down to the following
\begin{theorem}\label[theorem]{WkCnxFLew}
Let $\cnx\in\{\cnxp,\cnxm\}$ and $\Lg$ be a substructural logic such $(\Lg,\cnx)$ is connexive. Then $(\Lg,\cnx)$ is strongly connexive iff \eqref{efq} holds in $\Lg$. 
\end{theorem}

\section{Conclusion and future research}\label{sec: conclusion}
This work have been devoted to a preliminary investigation of term-definable connexive implications in Substructural Logics and their semantic features. 

We have shown that the Glivenko variety of $\FLe$ ($\FLei$, and $\FLew$) relative to Boolean algebras, provides a suitable semantical environment in which the connectives $\cnxm$ and $\cnxp$ serve as {\em bonafide} connexive implications, and {\em vice versa}. Consequently, any axiomatic extension of $\mathbf{G}_{\mathbf{FL}_\mathbf{e}}(\CPL)$ 
provides a full-fledged connexive logic once the $\{\land,\lor,\cdot,\cnxm,0,1\}$-fragment (and, assuming the axiom $\vdash\varphi\to\neg\neg 1$, the same holds replacing $\cnxm$ by $\cnxp$) is considered.

We have also argued that these logics provide a suitable framework (resp., semantics) for the \emph{logic of plausible inference} \`{a} la G.~Polya. As it has been remarked in Section \ref{sec:charwhencnxmdiscon}, focusing on $\cnxm$ does not result in a valuable loss of generality, since in most cases the connexivity of $\cnxm$ is equivalent to the connexivity of $\cnxp$ or, at least in the integral case, to the connexivity of \emph{any} binary operation in the interval $[\cnxp,\cnxm]$. See also Section \ref{sec: strength of cnxm}.

Finally, we have shown that, in our framework, strong connexivity is codified by \eqref{efq}. Therefore, in some cases, it can be expressed in the object language of a connexive logic by means of an anything but contra-classical inference rule.
We see this contribution as a stepping-stone for further investigations of both the connective $\cnxm$, in particular, and other (term-defined) implications in residuated lattices and substructural logics in general. For instance, it seems reasonable that the characterization results presented in the work may generalize to the non-commutative case; while other lines to be considered may be the following:
\begin{itemize}
\item Obviously, for any sub-variety $\mathsf{V}$ of the variety $\mathsf{In}\FLe$ of \emph{involutive} (i.e. satisfying $\neg\neg x \approx x$) \FLeA-algebras, neither $(\V,\cnxm)$ nor $(\V,\cnxp)$ are connexive. 
Therefore, an interesting task would be investigating term-definable connexive implications in involutive pointed commutative residuated lattices.
\item Due to results obtained so far, $\cnxm$ and $\cnxp$ fall short of behaving, in general, as suitable connexive implications once an arbitrary \FLeA-algebra is considered. However, extending the line of research \cite{GherOrla, Pizzi91}, it is reasonable to wonder if there are \emph{modal expansions} of \FLeA-algebras for which
full-fledged connexive implications can be term-defined.

\item Along the stream of research inaugurated by \cite{Wansing2007}, expansions of residuated lattices with connexive implications and/or unary operations $\sim$ which satisfy connexive theses together with $/$ and $\backslash$ would be worth of investigation.

\item As it has been pointed out in Section \ref{sec: weak connex}, in some cases (e.g. for extensions of $\mathbf{FL}_\mathbf{ew}$), weak connexivity is equivalent to (strong!) proto-connexivity. However, as witnessed by Example \ref{example:weak connexivity does not imply conn}, the former concept is properly weaker in the general framework, even for the well performing connective $\cnxm$.  Therefore, it naturally rises the question if a characterization of algebras admitting weakly connexive implications like the ones investigated in this paper can be provided.
\end{itemize}
And lastly there is, of course, the naturally risen question of if there is a cut-free sequent calculus for at least the $\{\land,\lor,\cdot,\cnxm,0,1\}$-fragment of $\mathbf{G}_{\mathbf{FL}_\mathbf{e}}(\CPL)$ [$\mathbf{G}_{\mathbf{FL}_\mathbf{ei}}(\CPL)$ or $\mathbf{G}_{\mathbf{FL}_\mathbf{ew}}(\CPL)$]. 
The problem is obviously strictly connected to finding a cut-free sequent calculus whose equivalent algebraic semantics are $\GLBA{\FLe}$, $\GLBA{\FLei}$ and $\GLBA{\FLew}$ which, to the best of our knowledge, is still missing.

\subsection*{Acknowledgements}
The authors gratefully acknowledge Antonio Ledda, Hitoshi Omori, Francesco Paoli, Adam P\v{r}enosil, and Sara Ugolini for the insightful discussions on the subjects of the present work. 

\end{document}